\documentclass[journal,twoside,web]{ieeecolor}
\usepackage{cite}		

\usepackage{generic}

\usepackage{amsmath,amssymb,amsfonts}
\usepackage[scr = dutchcal]{mathalpha}
\usepackage{graphicx}
\usepackage{textcomp}
\usepackage{stmaryrd}
\usepackage{tabularx}
\usepackage{multirow}
\usepackage{float}

\usepackage{color}

\newcommand{\mbs}[1]{\boldsymbol{#1}}

\newcommand{\mbf}[1]{\mathbf{ #1}}
\newcommand{\tbf}[1]{\textbf{#1}}

\newcommand{\mcl}[1]{\mathcal{#1}}
\newcommand{\mscr}[1]{\mathscr{#1}}

\newcommand{\R}{\mathbb{R}}
\newcommand{\N}{\mathbb{N}}

\newcommand{\norm}[1]{\left\lVert{#1}\right\rVert}
\newcommand{\ip}[2]{\left\langle #1, #2 \right\rangle}

\newcommand{\bmat}[1]{\begin{bmatrix}#1\end{bmatrix}}

\newcommand{\smallbmat}[1]{\left[\scriptsize\begin{smallmatrix}
		#1\end{smallmatrix} \right]}

\newcommand{\smallmat}[1]{\scriptsize\begin{smallmatrix}
		#1\end{smallmatrix}}

\newcommand{\srbmat}[1]{\small\left[\!\!\!\begin{array}{r}#1\end{array}\!\!\!\right]}
\newcommand{\slbmat}[1]{\small\left[\!\!\!\begin{array}{l}#1\end{array}\!\!\!\right]}

\newtheorem{thm}{Theorem}
\newtheorem{defn}[thm]{Definition}
\newtheorem{lem}[thm]{Lemma}
\newtheorem{prop}[thm]{Proposition}
\newtheorem{cor}[thm]{Corollary}

\let\bl\bigl
\let\bbl\Bigl
\let\bbbl\biggl
\let\bbbbl\Biggl
\let\br\bigr
\let\bbr\Bigr
\let\bbbr\biggr
\let\bbbbr\Biggr

\newcommand{\enn}[1]{\text{n}_{\text{#1}}}

\floatstyle{ruled}
\newfloat{block}{thbp}{lop}
\floatname{block}{Block}

\newcommand\Resize[2]{\resizebox{#1}{!}{\mbox{\ensuremath{\displaystyle #2}}}}

\title{\LARGE \bf
 	Representation of PDE Systems with Delay and Stability
 	Analysis using Convex Optimization -- Extended Version
}

\author{Declan S. Jagt, Matthew M. Peet %
	\thanks{\tbf{Acknowledgement:} This work was supported by National Science Foundation grant CMMI-1935453.} %
}

\begin{document}

\maketitle
\thispagestyle{empty}
\pagestyle{empty}

\begin{abstract}
	Partial Integral Equations (PIEs) have been used to represent both systems with delay and systems of Partial Differential Equations (PDEs) in one or two spatial dimensions. In this paper, we show that these results can be combined to obtain a PIE representation of any suitably well-posed 1D PDE model with constant delay. In particular, we represent these delayed PDE systems as coupled systems of 1D and 2D PDEs, obtaining a PIE representation of both subsystems. Taking the feedback interconnection of these PIE subsystems, we then obtain a 2D PIE representation of the 1D PDE with delay. 
	Next, based on the PIE representation, we formulate the problem of stability analysis as convex optimization of positive operators which can be solved using the PIETOOLS software suite. We apply the result to PDE examples with delay in the state and boundary conditions.
\end{abstract}


\section{INTRODUCTION}
	
We consider the problem of analysis of coupled systems of Ordinary Differential Equations (ODEs) and Partial Differential Equations (PDEs). Such ODE-PDE systems are frequently used to model physical processes, modeling the dynamics on the interior of the domain with the PDE, and the dynamics at the boundaries using the ODE.

In both modeling and control of PDE systems, the evolution of the system often depends on the internal state of the system at earlier points in time, giving rise to delays in different components of the model. For example, these delays may be inherent to the dynamics of the system itself, appearing within the PDE (sub)system, as in the following wave equation,
\begin{align*}
	\mbf{u}_{tt}(t,x)&=\mbf{u}_{xx}(t-\tau,x),\\
	\mbf{u}(t,1)&=-\mbf{u}(t,x),\qquad \mbf{u}(t,0)=0.
\end{align*}
Alternatively, delay may occur in the interaction between coupled systems, explicitly appearing in the Boundary Conditions (BCs) of the PDE,
\begin{align*}
	\dot{w}(t)&=-w(t)+\mbf{u}_{x}(t,1),\\
	\mbf{u}_{tt}(t,x)&=\mbf{u}_{xx}(t,x),\\
	\mbf{u}_{t}(t,1)&=-w(t-\tau),\qquad \mbf{u}(t,0)=0.
\end{align*}
or defining the dynamics at the boundary of the domain,
\begin{align*}
	\dot{w}(t)&=-w(t-\tau)+\mbf{u}_x(t,1),\\
	\mbf{u}_{tt}(t,x)&=\mbf{u}_{xx}(t,x),\\
	\mbf{u}_{t}(t,1)&=-w(t),\qquad \mbf{u}(t,0)=0.
\end{align*}
In each case, the presence of delays naturally complicates analysis of solution properties such as stability of the system, as at any time $t\geq 0$, the state of the system involves not only the current value of the state $\mbf{u}(t)$, but also the value of the ODE and PDE states $w(s)$ and $\mbf{u}(s)$ at all $s\in[t-\tau,t]$.

To verify stability of PDEs with delay, one common approach involves testing for existence of a positive definite functional $V$ that decays along solutions to the system -- i.e. a Lyapunov-Krasovskii Functional (LKF)~\cite{kolmanovskii1999DDEs}. In particular, for a delayed PDE with state $\mbf{u}(t)$ and delayed state $\mbs{\phi}(t)$ defined by $\mbs{\phi}(t,s):=\mbf{u}(t-s\tau)$ for $s\in[0,1]$, the challenge of proving stability then becomes that of finding a functional $V(\mbf{u},\mbs{\phi})$ that satisfies $V(0)=0$, $V(\mbf{u},\mbs{\phi})> 0$ for $(\mbf{u},\mbs{\phi})\neq 0$, and $\dot{V}(\mbf{u}(t),\mbs{\phi}(t))\leq 0$ along all solutions $(\mbf{u},\mbs{\phi})$ to the system.
In practice, a candidate LKF $V> 0$ is usually fixed a priori, often as some variation on the energy functional $V(\mbf{u},\mbs{\phi})=\|\mbf{u}\|^2+\|\mbs{\phi}\|^2$, and then proven to decay along solutions to a system of interest. 
Although stability properties of a variety of PDEs with delay have been proven this way, including for heat and wave equations with both time-varying and constant delay~\cite{nicaise2009stability_waveeq_heateq_BCdelay,ammari2010stabilization_waveeq}, 
results obtained in this manner are difficult to extend to other systems.
Specifically, a LKF that certifies stability for one system may not be valid for another, and identifying a suitable candidate LKF for a given system requires significant insight.

To test existence of LKFs for more general PDEs with delay, a cone of positive candidate functionals $V>0$ is often parameterized by positive definite matrices $P\succ 0$. The challenge in testing stability then becomes that of enforcing decay of the functionals, $\dot{V}\leq 0$, as a Linear Matrix Inequality (LMI), $Q\preceq 0$,
which can be efficiently solved using semidefinite programming.
Unfortunately, enforcing $\dot{V}\leq 0$ along solutions to a PDE with delay as an LMI is complicated by the fact that PDE dynamics are defined by (unbounded) differential operators, and that solutions are constrained to satisfy BCs.
As such, most prior work in this field focuses only on specific PDEs with delay, exploiting the structure of the PDE (parabolic, hyperbolic, elliptic) and the type of BCs (Dirichlet, Neumann, Robin) to enforce $\dot{V}\leq 0$.
For example, stability tests for heat and wave equations were derived in~\cite{fridman2009stabilityDPDE}, using the Wirtinger inequality to prove LMI constraints for negativity $\dot{V}\leq 0$. Using a similar approach, LMIs for stability of linear and semi-linear diffusive PDEs with delay were derived in~\cite{wang2014PDDEstability,wang2018PDDEstability,solomon2015PDDEstability}, as well as for reaction-diffusion systems with delayed boundary inputs in~\cite{lhachemi2021outputfeedackstabilization}. 
	
The disadvantage of these approaches, however, is that the results are again valid only for a restricted class of systems, and rely on the use of specific inequalities (e.g. Wirtinger, Jensen, Poincar\'e) to enforce $\dot{V}\leq 0$. Extending these results to even slightly different models, then, may require significant expertise from the user.

In this paper, we propose an alternative, LMI-based method for testing stability of a general class of linear ODE-PDE systems with fixed, constant delay, by representing them as Partial Integral Equations (PIEs). A PIE is an alternative representation of linear ODE-PDE systems, taking the form
\begin{align*}
	\mcl{T}\mbf{v}_{t}(t)&=\mcl{A}\mbf{v}(t),
\end{align*}
where the operators $\{\mcl{T},\mcl{A}\}$ are Partial Integral (PI) operators. In \cite{peet2021PIE_representation} and \cite{jagt2021PIEArxiv}, it was shown that the sets of 1D and 2D PI operators form *-algebras, meaning that the sum, composition, and adjoint of such PI operators is a PI operator as well. As such, parameterizing Lyapunov functionals $V(\mbf{v}(t))=\ip{\mcl{T}\mbf{v}(t)}{\mcl{P}\mcl{T}\mbf{v}(t)}$ by PI operators $\mcl{P}\succ 0$, the decay condition $\dot{V}\leq 0$ along solutions to the PIE can be enforced as a Linear PI Inequality (LPI)
\begin{align}\label{eq:stability_LPI_intro}
	\mcl{T}^*\mcl{P}\mcl{A}+\mcl{A}^*\mcl{P}\mcl{T}\leq 0,
\end{align}
Such LPIs constitute a specific class of linear operator inequalities (introduced for stability analysis of PDEs with delay in~\cite{fridman2009stabilityDPDE}), wherein the operator variable $\mcl{P}$ has the structure of a PI operator. Since the fundamental state $\mbf{v}(t)\in L_2$ in the PIE representation is not constrained by e.g. BCs, these LPI constraints need only be enforced on $L_{2}$.
Parameterizing positive PI operators $\mcl{P}\succeq 0$ by positive matrices $P\succeq0$, then, LPIs can be readily tested as LMIs, allowing problems of stability analysis~\cite{peet2021PIE_representation}, optimal control~\cite{shivakumar2020PIE_duality}, and optimal estimation~\cite{wu2023PIE_estimation} to be solved using convex optimization.

In~\cite{peet2021DDEs_PIEs}, it was shown that a general class of linear Delay Differential Equations (DDEs) can be equivalently represented as PIEs. Similarly, in~\cite{shivakumar2022GPDE_Arxiv}, it was shown that any suitably well-posed PDE system without delay can also be equivalently represented as a PIE. However, constructing a PIE representation for 1D PDE systems with delay is complicated by the fact that the delayed state $\mbs{\phi}(t,s,x)=\mbf{u}(t-s,x)$ in this case varies in two spatial variables. To address this problem, in this paper, we decompose the delayed PDE into a feedback interconnection of a 1D PDE and a 2D transport equation, where the interconnection signals are infinite-dimensional. We prove that each of these subsystems can be equivalently represented as an associated PIE with infinite-dimensional inputs and outputs, extending prior work on PIE input-output systems to the case of infinite-dimensional inputs and outputs. Next, we consider the feedback interconnection of PIEs with infinite-dimensional inputs and outputs, and derive formulae for the resulting closed-loop PIE. Finally, paramaterizing a LKF by PI operators, we establish stability conditions expressed as a convex optimization program, subject to LPI constraints. These LPIs are then converted to semidefinite programming problems using the PIETOOLS software package and tested on several examples of delayed PDE systems.

\section{Problem Formulation}

\subsection{Notation}

For a given domain $\Omega\subset\R^d$, let $L_2^n[\Omega]$ and $L_{\infty}^{n}[\Omega]$ denote the sets of $\R^n$-valued square-integrable and bounded functions on $\Omega$, respectively, where we omit the domain when clear from context. Define intervals $\Omega_{a}^{b}:=[a,b]$ and $\Omega_{c}^{d}:=[c,d]$, and let $\Omega_{ac}^{bd}:=\Omega_{a}^{b}\times\Omega_{c}^{d}$. 
For $\text{k}=(k_1,k_2)\in\N^2$, define Sobolev subspaces $H_{k_1}^n[\Omega_{a}^{b}]$ and $H_{\text{k}}^{n}[\Omega_{ab}^{cd}]$ of $L_2^{n}$ as
\begin{align*}
	&H_{k_1}^{n}[\Omega_{a}^{b}]=\bl\{\mbf{v}\mid \partial_x^{\alpha}\mbf{v}\!\in\! L_2^n[\Omega_{a}^{b}],\ \forall \alpha \in\N:\! \alpha\!\leq k_1\br\},	\\
	&H_{\text{k}}^{n}[\Omega_{ab}^{cd}]=\bbl\{\!\mbf{v}\!\ \bbr\rvert\  \partial_x^{\alpha_1}\partial_s^{\alpha_2}\mbf{v}\!\in\! L_2^n[\Omega_{ab}^{cd}],\ \forall \bl[{\smallmat{\alpha_1\\\alpha_2}}\br] \!\in\!\N^2:\! \smallmat{\alpha_1\leq k_1\\\alpha_2\leq k_2}\bbr\}.
\end{align*}

\subsection{Objectives and Approach}

\begin{figure*}[th!]
	\centering
	\hspace*{-0cm}\includegraphics[width=1.0\textwidth]{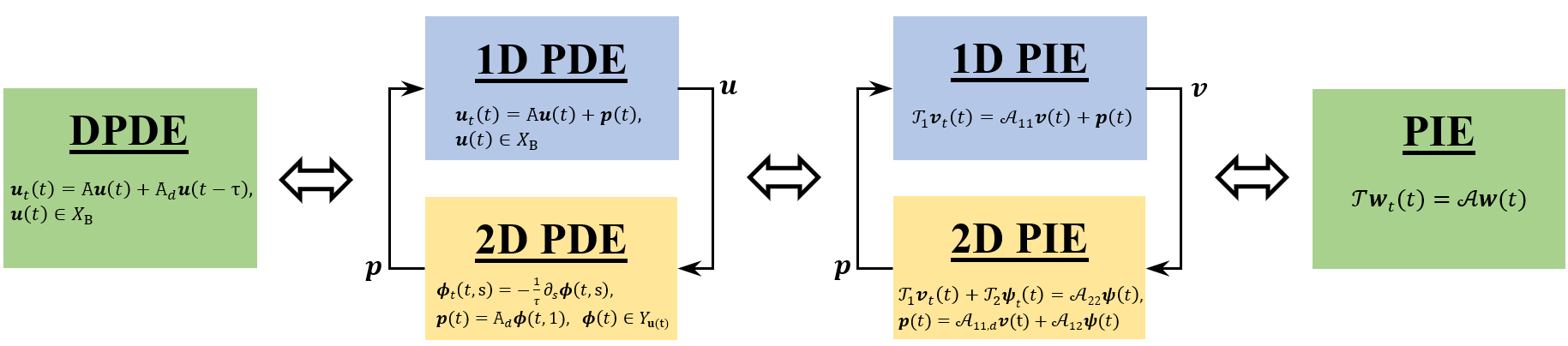}
	\caption{Block diagram representation of delayed PDE~\eqref{eq:DPDE_N2}. In Section~\ref{sec:PIE}, we expand the Delayed PDE (DPDE) as the interconnection of a 1D and 2D PDE. In Subsection~\ref{sec:PIE:1D} and Subsection~\ref{sec:PIE:1D}, we prove that the 1D and 2D subsystems can be equivalently represented as PIEs, respectively. Finally, In Section~\ref{sec:PIE_interconnection}, we prove that the feedback interconnection of these PIEs can be represented as a PIE as well, allowing us to test stability of the DPDE using the approach presented in Section~\ref{sec:stability}.}
	\label{fig:DPDE2PIE}
\end{figure*}

In this paper, we propose a framework for testing exponential stability of linear, 1D, 2nd order PDEs, with delay, focusing primarily on systems with delay in the dynamics. Specifically, we consider a delayed PDE of the form
\begin{align}\label{eq:DPDE_N2}
	&\mbf{u}_{t}(t,x)=A(x)\slbmat{\mbf{u}(t,x)\\\mbf{u}_{x}(t,x)\\ \mbf{u}_{xx}(t,x)} + A_{d}(x)\slbmat{\mbf{u}(t-\tau,x)\\\mbf{u}_{x}(t-\tau,x)\\ \mbf{u}_{xx}(t-\tau,x)},	\\
	&\mbf{u}(t)\in X_{B}[\Omega_{a}^{b}],\hspace*{3.0cm} t\geq 0,~x\in\Omega_{a}^{b},	\notag
\end{align}
where $A,A_{d}\in L_{\infty}^{n\times 3n}[\Omega_{a}^{b}]$, and where the PDE domain $X_{B}$ is constrained by boundary conditions and continuity constraints, and is defined by a matrix $B\in\R^{2n\times 4n}$ as
\begin{align}\label{eq:Xset}
	X_{B}[\Omega_{a}^{b}] := \left\{
	\mbf{u}\in H_{2}^{n}[\Omega_{a}^{b}] ~\bbbr\rvert~
	\small
	B\smallbmat{\mbf{u}(a)\\\mbf{u}(b)\\\mbf{u}_{x}(a)\\\mbf{u}_{x}(b)}
	= 0
	\right\},
\end{align}
where $B$ must be of full row-rank, defining sufficient and independent boundary conditions (see also Sec.~3.2 in~\cite{peet2021PIE_representation}).
In order to test stability of the delayed PDE~\eqref{eq:DPDE_N2}, we will first derive an equivalent representation of the system as a Partial Integral Equation (PIE), using the approach illustrated in Figure~\ref{fig:DPDE2PIE}. In particular, we take the following three steps:\\
\indent 1. First, in Section~\ref{sec:PIE}, we represent the Delayed PDE (DPDE) as the interconnection of a 1D PDE and a 2D PDE. \\
\indent 2. Then, in Subsection~\ref{sec:PIE:1D} and~\ref{sec:PIE:2D}, we derive equivalent 1D and 2D PIE representations of the 1D and 2D PDE subsystems, respectively. \\
\indent 3. Next, in Section~\ref{sec:PIE_interconnection}, we prove that the feedback interconnection of PIEs can be represented as a PIE as well, and take the interconnection of the 1D and 2D PIEs to obtain a PIE representation for the DPDE. 

The resulting PIE representation will be of the form
\begin{align}\label{eq:PIE_N2}
	\mcl{T}\mbf{w}_{t}(t,s,x)&=\mcl{A}\mbf{w}(t,s,x),	&	(s,x)&\in\Omega_{0a}^{1b},
\end{align}
wherein the state $\mbf{w}(t)\in L_2[\Omega_{a}^{b}]\times L_2[\Omega_{0a}^{1b}]$ is free of boundary conditions, and where the operators $\mcl{T}$ and $\mcl{A}$ are Partial Integral (PI) operators, defined as in Block~\ref{block:Top_Aop}. In Section~\ref{sec:stability}, we show that given such a PIE representation of a system, we can test stability of the system by solving a Linear PI operator Inequality (LPI) optimization program, defined by the operators $\mcl{T}$ and $\mcl{A}$. In particular, using this approach, we can test stability of the delayed PDE~\eqref{eq:DPDE_N2} as follows.
\begin{prop}\label{prop:stability_PDDE}
	Let $A,A_{d}\in L_{\infty}^{n\times 3n}$, $B\in\R^{2n\times 4n}$ and $\tau>0$ define a delayed PDE as in~\eqref{eq:DPDE_N2}. Define PI operators $\mcl{T},\mcl{A}$ as in Block~\ref{block:Top_Aop}. Suppose that there exist constants $\epsilon,\alpha>0$ and a PI operator $\mcl{P}$ such that $\mcl{P}=\mcl{P}^*$, $\mcl{P}\succeq \epsilon^2 I$, and
	\begin{align}
		\mcl{A}^*\mcl{P}\mcl{T}+\mcl{T}^*\mcl{P}\mcl{A}\preceq -2\alpha\mcl{T}^*\mcl{P}\mcl{T}.
	\end{align}
	Finally, let $\zeta=\sqrt{\|\mcl{P}\|_{\mcl{L}}}$.
	Then, for any solution $\mbf{u}$ to the PDE~\eqref{eq:DPDE_N2} with $\mbs{\phi}(t,s)=\mbf{u}(t-s\tau)$ for $s\in[0,1]$, we have
	\begin{align*}
		\norm{\slbmat{\mbf{u}(t)\\\mbs{\phi}(t)}}_{\text{Z}}\leq \frac{\zeta}{\epsilon}\norm{\slbmat{\mbf{u}(0)\\\mbs{\phi}(0)}}_{\text{Z}} e^{-\alpha t},\qquad \forall t\geq 0,
	\end{align*}
	where $\norm{\smallbmat{\mbf{u}(t)\\\mbs{\phi}(t)}}_{\text{Z}}^2=\|\mbf{u}(t)\|_{L_2}^2 +\int_{0}^{1}\|\mbs{\phi}(t,s)\|_{L_2}^2 ds$.
\end{prop}
In the following sections, we show how we arrive at this result, explicitly proving it in Cor.~\ref{cor:stability_PDDE}. In Section~\ref{sec:other_delay}, we will also briefly consider how other systems with delay, such as PDEs with delay in the boundary conditions or delay in a coupled ODE, can be converted to equivalent PIEs as well, allowing stability of those systems to be tested using a similar approach. A more detailed discussion on deriving this PIE representation can be found in the appendices. In Section~\ref{sec:numerical_examples}, we apply the proposed methodology to numerically test stability of several delayed PDE systems.


\begin{block*}
	\small
	Define $\mcl{T}:=\bmat{\mcl{T}_{1}&0\\\mcl{T}_{1}&\mcl{T}_{2}}$, $\mcl{A}:=\bmat{\mcl{A}_{11}+\mcl{A}_{11,\text{d}}&\!\mcl{A}_{12}\\0&\!\mcl{A}_{22}}$, 
	where $\mcl{T}_{1}:=\mcl{P}_{\text{T}}$, $\mcl{A}_{11}:=\mcl{P}_{\text{A}}$, and  $\mcl{A}_{11,d}:=\mcl{P}_{\text{A}_{d}}$ are 3-PI operators (see Defn.~\ref{defn:3PI}), 
	and where for $\hat{\mbf{v}}\in L_2[\Omega_{0a}^{1b}]$, \\[-2.2em]
	\begin{flalign*}
		\hspace*{4.5cm}
		(\mcl{T}_{2}\hat{\mbf{v}})(s) &:=\int_{0}^{s}(\mcl{T}_{1}\hat{\mbf{v}}(\theta))d\theta,	
		&
		&(\mcl{A}_{12}\hat{\mbf{v}}):=\int_{0}^{1}(\mcl{A}_{11,\text{d}}\hat{\mbf{v}}(s))ds,
		&
		&(\mcl{A}_{22}\hat{\mbf{v}})(s):=-\frac{1}{\tau}(\mcl{T}_{1}\hat{\mbf{v}}(s)), \\[-2.6em]
	\end{flalign*}
	with parameters \\[-1.5em]
	\begin{align*}
		&\begin{array}{l}
			\text{T}:=\{0,T_{1},T_{2}\}	\\
			T_{1}(x,\theta):=(x-\theta)I_{n}+T_{2}(x,\theta),	\\
			T_{2}(x,\theta):=-K(x)(BH)^{-1}BQ(x,\theta),	\\
			K(x):=\bmat{I_{n}&(x-a)I_{n}},
		\end{array}		&
		&H\!:=\!\left[
		\begin{array}{ll}
			I_{n}&0\\I_{n}&(b-a)I_{n}\\0&I_{n}\\0&I_{n}
		\end{array}
		\right],	&
		&Q(x,\theta):=\!\left[\begin{array}{l}
			0\\(b-\theta)I_{n}\\0\\I_{n}
		\end{array}\right],	
		&
		&\begin{array}{l}
			A_{0}(x):=A(x)\smallbmat{0\\0\\I_{n}},	\\		
			A_{d,0}(x):=A_{d}(x)\smallbmat{0\\0\\I_{n}},
		\end{array} \\[-1.6em]
	\end{align*}
	\begin{equation*}
		\begin{array}{l}
			\text{A}:=\{A_{0},A_{1},A_{2}\},	\\		
			\text{A}_{d}:=\{A_{d,0},A_{d,1},A_{d,2}\},
		\end{array}
		\hspace*{2.25cm}
		A_{j}(x,\theta):=A(x)\smallbmat{T_{j}(x,\theta)\\\partial_{x}T_{j}(x,\theta)\\0},\qquad
		A_{d,j}(x,\theta):=A_{d}(x)\smallbmat{T_{j}(x,\theta)\\\partial_{x}T_{j}(x,\theta)\\0},
		\qquad j\in\{1,2\}.
	\end{equation*}
	\caption{Operators $\mcl{T}$ and $\mcl{A}$ defining the PIE~\eqref{eq:PIE_N2} associated to the DPDE~\eqref{eq:DPDE_N2} with BCs as in~\eqref{eq:Xset}}\label{block:Top_Aop} 
\end{block*}

\section{A PIE Representation of Delayed PDEs}\label{sec:PIE}


In order to test stability of the DPDE~\eqref{eq:DPDE_N2}, we first derive a representation of this system wherein we model the delay using a transport equation. 
In particular, let $\mbs{\phi}(t,s)$ represent $\mbf{u}(t-s\tau)$ for $s\in[0,1]$. Then, $\mbf{u}(t)$ satisfies the PDE~\eqref{eq:DPDE_N2} if and only if the augmented state $(\mbf{u}(t),\mbs{\phi}(t))$ satisfies
\begin{align}\label{eq:PDDE_expanded_N2}
	\mbf{u}_{t}(t)&=\text{M}_{A} \slbmat{\mbf{u}(t)\\\mbf{u}_{x}(t)\\\mbf{u}_{xx}(t)} +\text{M}_{A_{\text{d}}} \slbmat{\mbs{\phi}(t,1)\\\mbs{\phi}_{x}(t,1)\\\mbs{\phi}_{xx}(t,1)},
	&	\mbf{u}(t)&\in X_{B},	\\
	\mbs{\phi}_{t}(t)&=-(1/\tau)\mbs{\phi}_{s}(t), 	& \mbs{\phi}(t)&\in Y_{\mbf{u}(t)},	\notag
\end{align}
where $\text{M}_{A}$ denotes the multiplier operator associated to $A\in L_{\infty}$, so that $(\text{M}_{A}\mbf{q})(x)=A(x)\mbf{q}(x)$ for $\mbf{q}\in L_2$, and
where we define the domain of the delayed state $\mbs{\phi}(t)$ as
\begin{equation}\label{eq:Yset_PDE}\Resize{0.91\linewidth}{
		Y_{\mbf{u}}\!:=\!\bbl\{\!
		\mbs{\phi}\in H_{(1,2)}^{n}[\Omega_{0a}^{1b}] \,\bbr\rvert\, \mbs{\phi}(0,x)=\mbf{u}(x),~\mbs{\phi}(s,.)\in X_{B} \!\bbr\}.}
\end{equation}
Given this expanded representation of the system, we define solutions to the DPDE~\eqref{eq:DPDE_N2} as follows.
\begin{defn}[Solution to the DPDE]
	For a given initial state $(\mbf{u}_{0},\mbs{\phi}_{0})\in X_{B}\times Y_{\mbf{u}_{0}}$, we say that $(\mbf{u},\mbs{\phi})$ is a solution to the DPDE defined by $\{A,A_{\text{d}},B,\tau\}$ if $(\mbf{u},\mbs{\phi})$ is Frech\'et differentiable, $(\mbf{u}(0),\mbs{\phi}(0))=(\mbf{u}_{0},\mbs{\phi}_{0})$, and for all $t\geq 0$,
	$\bl(\mbf{u}(t),\mbs{\phi}(t)\br)$ satisfies~\eqref{eq:PDDE_expanded_N2}.
\end{defn}

Although the expanded representation in~\eqref{eq:PDDE_expanded_N2} no longer involves explicit time-delay in the state, stability analysis in this representation is still complicated by the auxiliary constraints $\bl(\mbf{u}(t),\mbs{\phi}(t)\br)\in X_{B}\times Y_{\mbf{u}(t)}$. Therefore, in the following subsections, we will separately consider the dynamics of $\mbf{u}(t)$ and $\mbs{\phi}(t)$, representing these dynamics in an equivalent format free of auxiliary constraints -- as PIEs.

%
%



\subsection{A PIE Representation of 1D PDEs}\label{sec:PIE:1D}

Consider the 1D subsystem of the coupled PDE in~\eqref{eq:PDDE_expanded_N2},
\begin{equation}\label{eq:PDE_1D}
	\mbf{u}_{t}(t)=\text{M}_{A} \smallbmat{\mbf{u}(t)\\\mbf{u}_{x}(t)\\\mbf{u}_{xx}(t)} +\mbf{p}(t),\qquad \mbf{u}(t)\in X_{B},
\end{equation}
where now $\mbf{p}(t)=\text{M}_{A_{\text{d}}}\smallbmat{\mbs{\phi}(t,1)\\\mbs{\phi}_{x}(t,1)\\\mbs{\phi}_{xx}(t,1)}\in L_{2}^{n}[\Omega_{a}^{b}]$ is considered to be an input. In this system, the state $\mbf{u}(t)\in X_{B}\subseteq H_2^{n}[\Omega_{a}^{b}]$ at any time $t\geq 0$ is only second-order differentiable with respect to the spatial variable $x$. As such, the second-order derivative $\mbf{v}(t):=\mbf{u}_{xx}(t)\in L_2^{n}[\Omega_{a}^{b}]$ of the state does not have to satisfy any boundary conditions or continuity constraints, and we refer to $\mbf{v}(t)$ as the \textit{fundamental state} associated to the PDE.
In this subsection, we will derive an equivalent representation of the 1D subsystem in~\eqref{eq:PDE_1D} in terms of this fundamental state $\mbf{v}(t)$, as a PIE. To this end, we first recall the definition of a 3-PI operator.


\begin{defn}\label{defn:3PI}[3-PI Operators ($\Pi_{3}$)]
	For $m,n\in\N$, define
	\begin{equation*}
		\mcl{N}_{3}^{m\times n}[\Omega_{a}^{b}]\!:=\!L_2^{m\times n}[\Omega_{a}^{b}] \!\times\! L_2^{m\times n}[\Omega_{a}^{b}\!\times\!\Omega_{a}^{b}] \times L_2^{m\times n}[\Omega_{a}^{b}\!\times\!\Omega_{a}^{b}].
	\end{equation*}
	Then, for given parameters $\text{R}:=\{R_0,R_1,R_2\}\in\mcl{N}_{3}$, we define the operator $\mcl{R}=\mcl{P}_{\text{R}}$ for $\mbf{u}\in L_2^{n}[\Omega_{a}^{b}]$ as
	\begin{equation*}\Resize{\linewidth}{
			\bl(\mcl{R}\mbf{u}\br)(x)= R_0(x)\mbf{u}(x) + \!\int_{a}^{x}\!\! R_1(x,\theta)\mbf{u}(\theta)d\theta 	
			+\! \int_{x}^{b}\!\! R_2(x,\theta)\mbf{u}(\theta)d\theta.}
	\end{equation*}
	We say $\mcl{R}\in\Pi_{3}$ if $\mcl{R}:=\mcl{P}_{\text{R}}$ for some $\text{R}\in\mcl{N}_{3}$. For convenience, we say that $\mcl{R}$ is a 3-PI operator.
\end{defn}
Defining 3-PI operators in this manner, it has been shown that $\Pi_{3}$ forms a *-algebra -- i.e. is closed under summation, composition, scalar-multiplication and adjoint with respect to $L_2$~\cite{shivakumar2022GPDE_Arxiv}. Moreover, under mild assumptions on the boundary conditions $B$, we can define a continuous, bijective map $\mcl{T}_{1}:L_2^{n}\to X_{B}$ from the fundamental to the PDE state space as a 3-PI operator, as shown in the following result from~\cite{peet2021PIE_representation}.
\begin{lem}\label{lem:Tmap_1D}
	Let $X_{B}$ be as defined in~\eqref{eq:Xset} for some $B\in\R^{2n\times 4n}$ which is such that the matrix $BH\in\R^{2n\times 2n}$ is invertible with $H$ as in Block~\ref{block:Top_Aop}.
	If $\mcl{T}_{1}\in\Pi_{3}$ is as defined in Block~\ref{block:Top_Aop}, then, for every $\mbf{u}\in X_{B}$ and $\mbf{v}\in L_2^{n}$,
	\begin{align*}
		\mbf{u}&=\mcl{T}_{1} (\partial_{x}^2 \mbf{u}),	&	&\text{and}	&
		\mbf{v}&=\partial_{x}^2 (\mcl{T}_{1}\mbf{v}).
	\end{align*} 
\end{lem}
\medskip
\begin{proof}
	Defining the parameters $K,H,Q,T_{j}$ as in Block~\ref{block:Top_Aop}, and using the fundamental theorem of calculus and Cauchy's formula for repeated integration, we can show that
	\begin{align*}
		\mbf{u}&=\text{M}_{K}\smallbmat{\mbf{u}(a)\\\mbf{u}_{x}(a)} +\mcl{P}_{\{0,T_{1}-T_{2},0\}}\mbf{u}_{xx},	\\
		\smallbmat{\mbf{u}(a)\\\mbf{u}(b)\\\mbf{u}_{x}(a)\\\mbf{u}_{x}(b)}&=H\smallbmat{\mbf{u}(a)\\\mbf{u}_{x}(a)} +\mcl{P}_{\{0,Q,Q\}}\mbf{u}_{xx},	& \forall \mbf{u}\in H_{2}^{n}[\Omega_{a}^{b}].
	\end{align*}
	Imposing the boundary conditions $B\smallbmat{\mbf{u}(a)\\\mbf{u}(b)\\\mbf{u}_{x}(a)\\\mbf{u}_{x}(b)}=0$, and performing standard algebraic manipulations, we can then show that for all $\mbf{u}\in X_{B}$,
	\begin{equation*}
		\mbf{u}=\bl[\mcl{P}_{\{0,T_{1}-T_{2},0\}}-\text{M}_{K}(BH)^{-1}B\mcl{P}_{\{0,Q,Q\}}\br]\mbf{u}_{xx}=\mcl{T}_{1}\mbf{u}_{xx}.
	\end{equation*}
	Conversely, using the Leibniz integral rule, we can show that $\partial_{x}^2\circ\mcl{T}_{1}=I$, so that $\mbf{v}=\partial_{x}^2(\mcl{T}_{1}\mbf{v})$ for all $\mbf{v}\in L_2$. A full proof is given in~\cite{peet2021PIE_representation}.
\end{proof}
Lem.~\ref{lem:Tmap_1D} proves that, given sufficiently well-posed boundary conditions, any $\mbf{u}\in X_{B}$ is uniquely defined by its highest-order partial derivative $\mbf{v}=\mbf{u}_{xx}\in L_2^{n}$ as $\mbf{u}=\mcl{T}_{1}\mbf{v}$. Using the Leibniz integral rule, we can then also express
\begin{equation*}\Resize{\linewidth}{
		\text{M}_{A}\smallbmat{\mbf{u}\\\mbf{u}_{x}\\\mbf{u}_{xx}}
		=\text{M}_{A}\smallbmat{\mcl{T}_{1}\mbf{u}_{xx}\\ \partial_{x}\mcl{T}_{1}\mbf{u}_{xx}\\ \mbf{u}_{xx}}
		=\text{M}_{A}\smallbmat{\mcl{P}_{\{0,T_{1},T_{2}\}}\mbf{v}\\\mcl{P}_{\{0,\partial_{x}T_{1},\partial_{x}T_{2}\}}\mbf{v}\\\mcl{P}_{\{I_{n},0,0\}}\mbf{v}}
		=\mcl{A}_{11}\mbf{v},}
\end{equation*}
for $\mcl{A}_{11}\in\Pi_{3}$ as in Block~\ref{block:Top_Aop}. It follows that $\mbf{u}(t)$ satisfies the PDE~\eqref{eq:PDE_1D} if and only if $\mbf{v}(t)=\mbf{u}_{xx}(t)$ satisfies the PIE
\begin{align}\label{eq:PIE_1D}
	\mcl{T}_{1}\mbf{v}_{t}(t)&=\mcl{A}_{11}\mbf{v}(t) +\mbf{p}(t), \qquad \mbf{v}(t)\in L_2^{n}[\Omega_{a}^{b}].
\end{align}
In particular, we have the following result.

\begin{lem}\label{lem:PIE_1D}
	Suppose that $A\in L_{\infty}^{n\times 3n}[\Omega_{a}^{b}]$ and $B\in\R^{2n\times 4n}$ satisfies the conditions of Lem.~\ref{lem:Tmap_1D}. Define operators $\mcl{T}_{1},\mcl{A}_{1}\in\Pi_{3}$ as in  Block~\ref{block:Top_Aop}. Then, for any given input $\mbf{p}(t)\in L_{2}^{n}[\Omega_{a}^{b}]$, $\mbf{v}$ is a solution to the PIE~\eqref{eq:PIE_1D} with initial state $\mbf{v}_{0}\in L_2^{n}[\Omega_{a}^{b}]$ if and only if $\mbf{u}=\mcl{T}_{1}\mbf{v}$ is a solution to the PDE~\eqref{eq:PDE_1D} with initial state $\mbf{u}_{0}=\mcl{T}_{1}\mbf{v}_{0}$. Conversely, 
	$\mbf{u}$ is a solution to the PDE~\eqref{eq:PDE_1D} with initial state $\mbf{u}_{0}\in X_{B}$ if and only if $\mbf{v}=\partial_{x}^2\mbf{u}$ is a solution to the PIE~\eqref{eq:PIE_1D} with initial state $\mbf{v}_{0}=\partial_{x}^2\mbf{u}_{0}$.

\end{lem}
\begin{proof}
	The proof follows substituting the relation $\mbf{u}(t)=\mcl{T}_{1}\mbf{v}(t)$ into the PDE~\eqref{eq:PDE_1D}. A full proof is given in e.g.~\cite{peet2021PIE_representation}.
\end{proof}


\subsection{A PIE Representation of 2D Transport Equations}\label{sec:PIE:2D}


Consider now the 2D subsystem of the coupled PDE in~\eqref{eq:PDDE_expanded_N2},
\begin{align}\label{eq:PDE_2D}
	\mbs{\phi}_{t}(t)&=-(1/\tau)\mbs{\phi}_{s}(t),	&	\mbs{\phi}(t)\in Y_{\mcl{T}_{1}\mbf{v}(t)},	\\
	\mbf{p}(t)&=\text{M}_{A_{d}}\smallbmat{\mbs{\phi}(t,1)\\\mbs{\phi}_{x}(t,1)\\\mbs{\phi}_{xx}(t,1)},	\notag
\end{align}
wherein we consider $\mbf{v}(t)=\mbf{u}_{xx}(t)\in L_{2}^{n}[\Omega_{a}^{b}]$ as an input, and $\mbf{p}(t)\in L_{2}^{n}[\Omega_{a}^{b}]$ as an output. 
Although a framework for constructing PIE representations for general 2D PDEs has been developed in~\cite{jagt2021PIEArxiv}, in this case, we can significantly simplify this construction by exploiting the structure of the 2D subsystem. In particular, by definition of the space $Y_{\mbf{u}(t)}$, any $\mbs{\phi}(t)\in Y_{\mbf{u}(t)}$ must satisfy the same boundary conditions as $\mbf{u}(t)$. As such, we can use the same operator $\mcl{T}_{1}$ as in Lem.~\ref{lem:Tmap_1D} to also express $\mbs{\phi}(t)$ in terms of its associated fundamental state $\mbs{\phi}_{sxx}(t)$, as shown in the following lemma.


\begin{lem}\label{lem:Tmap_2D}
	Let $Y_{\mbf{u}}$ be as defined in~\eqref{eq:Yset_PDE}, with the set $X_{B}$ as defined in~\eqref{eq:Xset} for some $B\in\R^{2n\times 4n}$ satisfying the conditions of Lem.~\ref{lem:Tmap_1D}.
	If $\mcl{T}_{1}\in\Pi_{3}$ and $\mcl{T}_{2}$ are as defined in Block~\ref{block:Top_Aop} and $\mbf{u}\in X_{B}$, then, for every $\mbs{\phi}\in Y_{\mbf{u}}$ and every $\mbs{\psi}\in L_2^{n}[\Omega_{0a}^{1b}]$,
	\begin{align*}
		\mbs{\phi}&=\mbf{u} +\mcl{T}_{2}(\partial_{s}\partial_{x}^2\mbs{\phi}),	&	&\text{and}	&
		\mbs{\psi}&=\partial_{s}\partial_{x}^2 (\mbf{u}+\mcl{T}_{2}\mbs{\psi}).
	\end{align*} 
\end{lem}
\smallskip
\begin{proof}
	Fix arbitrary $\mbf{u}\in X_{B}$ and $\mbs{\phi}\in Y_{\mbf{u}}$. By definition of the set $\mbs{\phi}\in Y_{\mbf{u}}$, we have $\mbs{\phi}(0)=\mbf{u}$ and $\mbs{\phi}(s)\in X_{B}$ for all $s\in[0,1]$. By Lemma~\ref{lem:Tmap_1D}, then, $\mbs{\phi}(s)=\mcl{T}_{1}(\partial_{x}^2 \mbs{\phi}(s))$ for all $s\in[0,1]$, implying that also
	\begin{equation*}
		\partial_{s}\mbs{\phi}(s)=\partial_{s}\mcl{T}_{1}(\partial_{x}^2 \mbs{\phi}(s))=\mcl{T}_{1}(\partial_{s}\partial_{x}^2\mbs{\phi}(s)).
	\end{equation*}
	Invoking the fundamental theorem of calculus, and using the definition of the operator $\mcl{T}_{2}$, it follows that
	\begin{align*}
		\mbs{\phi}(s)&=\mbs{\phi}(0)+\!\int_{0}^{s}\!\!\partial_{s}\mbs{\phi}(\theta)d\theta	\\
		&=\mbf{u} +\!\int_{0}^{s}\!\!\mcl{T}_{1}(\partial_{s}\partial_{x}^2\mbs{\phi}(\theta))d\theta	
		=\mbf{u} +\bl(\mcl{T}_{2}(\partial_{s}\partial_{x}^{2}\mbs{\phi})\br)(s).
	\end{align*}
	Now, fix arbitrary $\mbs{\psi}\in L_2^{n}[\Omega_{0a}^{1b}]$. Then, for all $s\in[0,1]$,
	\begin{equation*}\Resize{\linewidth}{
			\partial_{s}\partial_{x}^2\bl(\mbf{u}+(\mcl{T}_{2}\mbs{\psi})(s)\br)
			=\partial_{x}^{2}\partial_{s}\bbbl(\int_{0}^{s}\!\!\mcl{T}_{1}\bl(\mbs{\psi}(\theta)\br)d\theta\!\bbbr)
			=\partial_{x}^{2}\mcl{T}_{1}(\mbs{\psi}(s)).}
	\end{equation*}
	Here, by Lem.~\ref{lem:Tmap_1D}, $\partial_{x}^{2}\mcl{T}_{1}(\mbs{\psi}(s))=\mbs{\psi}(s)$ for all $s\in[0,1]$. 
\end{proof}

Using the relation $\mbs{\phi}(t)=\mcl{T}_{1}\mbf{v}(t)+\mcl{T}_{2}\mbs{\psi}(t)$ with $\mbs{\psi}(t)=\mbs{\phi}_{sxx}(t)$, and defining operators $\{\mcl{A}_{22},\mcl{A}_{11,d},\mcl{A}_{12}\}$ as in Block~\ref{block:Top_Aop}, we can show that if $(\mbs{\phi},\mbf{v},\mbf{p})$ satisfies the 2D transport PDE~\eqref{eq:PDE_2D}, then $(\mbs{\psi},\mbf{v},\mbf{p})$ satisfies the PIE
\begin{align}\label{eq:PIE_2D}
	&\mcl{T}_{1}\mbf{v}_{t}(t) +\mcl{T}_{2}\mbs{\psi}_{t}(t)=\mcl{A}_{22}\mbs{\psi}(t),	&	\mbs{\psi}(t)&\in L_{2}^{n}[\Omega_{0a}^{1b}],	\\
	&\mbf{p}(t)=\mcl{A}_{11,d}\mbf{v}(t) +\mcl{A}_{12}\mbs{\psi}(t).	\notag
\end{align}
In particular, we have the following result.

\begin{lem}\label{lem:PIE_2D}
	Suppose that $A_{d}\in L_{\infty}^{n\times 3n}[\Omega_{a}^{b}]$ and $\tau>0$, and that $B\in\R^{2n\times 4n}$ satisfies the conditions of Lem.~\ref{lem:Tmap_1D}. Define PI operators $\{\mcl{T}_{1},\mcl{T}_{2},\mcl{A}_{22},\mcl{A}_{11,d},\mcl{A}_{12}\}$ as in  Block~\ref{block:Top_Aop}. Then, for any given input $\mbf{v}(t)\in L_{2}[\Omega_{a}^{b}]$, $(\mbs{\psi},\mbf{p})$ solves the PIE~\eqref{eq:PIE_2D} with initial state $\mbs{\psi}_{0}\in L_2^{n}[\Omega_{0a}^{1b}]$ if and only if $\mbs{\phi}=\mcl{T}_{1}\mbf{v}+\mcl{T}_{2}\mbs{\psi}$ solves the PDE~\eqref{eq:PDE_2D} with initial state $\mbs{\phi}_{0}=\mcl{T}_{1}\mbf{v}(0)+\mcl{T}_{2}\mbs{\psi}_{0}$. Conversely, 
	$(\mbs{\phi},\mbf{p})$ solves the PDE~\eqref{eq:PDE_2D} with initial state $\mbs{\phi}_{0}\in Y_{\mcl{T}_{1}\mbf{v}(0)}$ if and only if $\mbs{\psi}=\partial_{s}\partial_{x}^2\mbs{\phi}$ solves the PIE~\eqref{eq:PIE_2D} with initial state $\mbs{\psi}_{0}=\partial_{s}\partial_{x}^2\mbs{\phi}_{0}$. 
\end{lem}

\begin{proof}
	Fix arbitrary $\mbs{\psi}_{0}\in L_{2}^{n}[\Omega_{0a}^{1b}]$ and $\mbf{v}(t)\in L_2^{n}[\Omega_{a}^{b}]$ for $t\geq 0$. Let $\mbs{\psi}(t)\in L_2[\Omega_{0a}^{1b}]$ and define $\mbs{\phi}(t)=\mcl{T}_{1}\mbf{v}(t)+\mcl{T}_{2}\mbs{\psi}(t)$. By Lemma~\ref{lem:Tmap_2D}, $\mbs{\phi}(t)\in Y_{\mcl{T}_{1}\mbf{v}(t)}$. In addition, it is clear that $\mbs{\psi}(0)=\mbs{\psi}_{0}$ if and only if $\mbs{\phi}(0)=\mcl{T}_{1}\mbf{v}(0)+\mcl{T}_{2}\mbs{\psi}_{0}$. Moreover, since $\mbf{v}(t)$ does not vary in $s\in[0,1]$,
	\begin{equation*}\Resize{\linewidth}{
			\mbs{\phi}_{s}(t)=\partial_{s}\mcl{T}_{2}\mbs{\psi}(t)
			=\partial_{s}\!\int_{0}^{s}\!\!\bl(\mcl{T}_{1}\mbs{\psi}\br)(t,\theta)d\theta =\mcl{T}_{1}\mbs{\psi}(t)=-\tau\mcl{A}_{22}\mbs{\psi}(t).}
	\end{equation*}
	where we invoke the definition of the operator $\mcl{A}_{22}=-\frac{1}{\tau}\mcl{T}_{1}$.
	It follows that
	\begin{equation}\label{eq:proof:lem:PIE_2D_1}
		\mcl{T}_{1}\mbf{v}_{t}(t)+\mcl{T}_{2}\mbs{\psi}_{t}(t)-\mcl{A}_{22}\mbs{\psi}(t)
		=\mbs{\phi}_{t}(t)+(1/\tau)\mbs{\phi}_{s}(t).
	\end{equation}
	Furthermore, using the Leibniz integral rule, we find that for any $\mbf{w}\in L_2^{n}[\Omega_{a}^{b}]$,
	{\small
		\begin{align*}
			\partial_{x}\bl(\mcl{T}_{1}\mbf{w}\br)
			&=\partial_{x}\int_{a}^{x}\!T_{1}(x,\theta)\mbf{w}(\theta)d\theta +\partial_{x}\int_{x}^{b}\!T_{2}(x,\theta)\mbf{w}(\theta)d\theta	\\
			&=T_{1}(x,x)\mbf{w}(x) +\int_{a}^{x}\!\partial_{x}T_{1}(x,\theta)\mbf{w}(\theta)d\theta	\\
			&\qquad-T_{2}(x,x)\mbf{w}(x) +\int_{x}^{b}\!\partial_{x}T_{2}(x,\theta)\mbf{w}(\theta)d\theta	\\
			&=\int_{a}^{x}\!\partial_{x}T_{1}(x,\theta)\mbf{w}(\theta)d\theta +\int_{x}^{b}\!\partial_{x}T_{2}(x,\theta)\mbf{w}(\theta)d\theta.
	\end{align*}}
	Noting that $I\circ\mcl{T}_{1}=\mcl{T}_{1}$ and $\partial_{x}^2\circ\mcl{T}_{1}=I$, and invoking  the definition of the operators $\mcl{A}_{11,d}$ and $\mcl{A}_{12}$, it follows that
	\begin{align}\label{eq:proof:lem:PIE_2D_2}
		&\bl(\mcl{A}_{11,d}\mbf{v}\br)(t,x) +\bl(\mcl{A}_{12}\mbs{\psi}\br)(t,x)	\\
		&=A_{d}(x)\left(\smallbmat{I\\\partial_{x}\\\partial_{x}^2}\bl(\mcl{T}_{1}\mbf{v}\br)(t,x) +\!\int_{0}^{1}\!\smallbmat{I\\\partial_{x}\\\partial_{x}^2}\!\bl(\mcl{T}_{1}\mbs{\psi}\br)(t,s,x)ds\right)	\notag\\
		&\hspace*{4.5cm}=A_{d}(x)\smallbmat{I\\\partial_{x}\\\partial_{x}^2}\mbs{\phi}(t,1,x).	\notag
	\end{align}
	By~\eqref{eq:proof:lem:PIE_2D_1} and~\eqref{eq:proof:lem:PIE_2D_2}, we conclude that $(\mbs{\psi}(t),\mbf{p}(t))$ satisfies the PIE~\eqref{eq:PIE_2D} if and only if $(\mbs{\phi}(t),\mbf{p}(t))$ satisfies the PDE~\eqref{eq:PDE_2D}.
	
	For the converse result, let $\mbs{\phi}_{0}\in Y_{\mcl{T}_{1}\mbf{v}(0)}$ and $\mbs{\phi}(t)\in Y_{\mcl{T}_{1}\mbf{v}(t)}$, and define $\mbs{\psi}_{0}=\partial_{s}\partial_{x}^2 \mbs{\phi}_{0}$ and $\mbs{\psi}(t)=\partial_{s}\partial_{x}^2 \mbs{\phi}(t)$. By Lem.~\ref{lem:Tmap_2D}, $\mbs{\phi}(t)=\mcl{T}_{1}\mbf{v}(t)+\mcl{T}_{2}\mbs{\psi}(t)$ and $\mbs{\phi}_{0}=\mcl{T}_{1}\mbf{v}(0)+\mcl{T}_{2}\mbs{\psi}_{0}$. By the first implication, it follows that $(\mbs{\phi},\mbf{p})$ is a solution to the PDE with initial state $\mbs{\phi}_{0}$ if and only if $(\mbs{\psi},\mbf{p})$ is a solution to the PIE with initial state $\mbs{\psi}_{0}$. 
\end{proof}

\section{Feedback Interconnection of PIEs}\label{sec:PIE_interconnection}

Having constructed a PIE representation of both the 1D and 2D subsystems of the PDE~\eqref{eq:PDDE_expanded_N2}, we now take the feedback interconnection of these PIE subsystems to obtain a PIE representation for the full delayed PDE. Specifically, we first prove that such a feedback interconnection of PIEs can itself be represented as PIE as well. To this end, consider a standardized PIE with input $\mbf{p}$ and output $\mbf{q}$ of the form
\begin{align}\label{eq:standard_PIE}
	\mcl{T}_{p}\mbf{p}_{t}(t)+\mcl{T}\mbf{v}_{t}(t)&=\mcl{A}\mbf{v}(t)+\mcl{B}\mbf{p}(t), & \mbf{v}(t)&\in \text{Z}^{\enn{v}}[\Omega],	\notag\\
	\mbf{q}(t)&=\mcl{C}\mbf{v}(t) +\mcl{D}\mbf{p}(t),		
\end{align}
where $\mcl{T}_{p}$ through $\mcl{D}$ are all PI operators, and where we define
\begin{equation}\label{eq:Zspace}
	\text{Z}^{\enn{}}[\Omega]:=\R^{n_{0}}\times L_2^{n_{1}}[\Omega_{1}]\times L_2^{n_{2}}[\Omega_{2}]\times L_2^{n_{3}}[\Omega_{1}\times\Omega_{2}]
\end{equation}
for $\Omega=\Omega_{1}\times\Omega_{2}\subseteq\R^{2}$ and $\enn{}\in\N^{4}$, so that we may use this format to express (coupled) 1D and 2D PDEs. We will allow the inputs $\mbf{p}(t)\in \text{Z}^{\enn{p}}$ and outputs $\mbf{q}(t)\in \text{Z}^{\enn{q}}$ to be distributed as well. We may collect the PI operators defining the system as $\mbf{G}:=\{\mcl{T},\mcl{T}_{p},\mcl{A},\mcl{B},\mcl{C},\mcl{D}\}$, writing $\mbf{G}\in\mbf{\Pi}^{(\enn{v},\enn{q})\times (\enn{v},\enn{p})}$. If $\enn{p}=\enn{q}=0$, we write $\mbf{G}:=\{\mcl{T},\mcl{A}\}\in\mbf{\Pi}^{\enn{v}\times \enn{v}}$.

\begin{defn}[Solution to the PIE]
	For a given input signal $\mbf{p}$ and initial state $\mbf{v}_{0}\in \text{Z}^{\enn{v}}$, we say that $(\mbf{v},\mbf{q})$ is a solution to the PIE defined by $\mbf{G}:=\{\mcl{T},\mcl{T}_{p},\mcl{A},\mcl{B},\mcl{C},\mcl{D}\}$ if $\mbf{v}$ is Frech\'et differentiable, $\mbf{v}(0)=\mbf{v}_{0}$, and for all $t\geq0$, $(\mbf{v}(t),\mbf{q}(t),\mbf{p}(t))$ satisfies Eqn.~\eqref{eq:standard_PIE}.
\end{defn}

By the composition and addition rules of PI operators, it follows that the feedback interconnection of two suitable PIEs can be represented as a PIE as well.


\begin{prop}[Interconnection of PIEs]\label{prop:PIE_interconnection}
	Let
	\begin{align*}
		\mbf{G}_{1}&:=\{\mcl{T}_{1},\mcl{T}_{p},\mcl{A}_{1},\mcl{B}_{p},\mcl{C}_{q},\mcl{D}_{qp}\}\in\mbf{\Pi}^{(\enn{1},\enn{q})\times (\enn{1},\enn{p})},	\\
		\mbf{G}_{2}&:=\{\mcl{T}_{2},\mcl{T}_{q},\mcl{A}_{2},\mcl{B}_{q},\mcl{C}_{p},0\}\in\mbf{\Pi}^{(\enn{2},\enn{p})\times (\enn{2},\enn{q})}
	\end{align*}
	and define $\mbf{G}:=\{\mcl{T},\mcl{A}\}\in \mbf{\Pi}^{\enn{v}\times \enn{v}}$ with $\enn{v}=\enn{1}+\enn{2}$ as
	\begin{equation*}\Resize{\linewidth}{
			\mcl{T}\!:=\!\bmat{\mcl{T}_{1} &\! \mcl{T}_{p}\mcl{C}_{p}\\ \mcl{T}_{q}\mcl{C}_{q} &\! \mcl{T}_{2}\!+\!\mcl{T}_{q}\mcl{D}_{qp}\mcl{C}_{p}}\!,	\quad
			\mcl{A}\!:=\!\bmat{\mcl{A}_{1} &\! \mcl{B}_{p}\mcl{C}_{p}\\ \mcl{B}_{q}\mcl{C}_{q} &\! \mcl{A}_{2}\!+\!\mcl{B}_{q}\mcl{D}_{qp}\mcl{C}_{p}}\!.}
	\end{equation*}
	Then, $\smallbmat{\mbf{v}\\\mbs{\psi}}$ solves the PIE defined by $\mbf{G}$ with initial values $\smallbmat{\mbf{v}_{0}\\\mbs{\psi}_{0}}$ if and only if $(\mbf{v},\mbf{q})$ and $(\mbs{\psi},\mbf{p})$ solve the PIEs defined by $\mbf{G}_{1}$ and $\mbf{G}_{2}$ with initial values $\mbf{v}_{0}$ and $\mbs{\psi}_{0}$ and inputs $\mbf{p}$ and $\mbf{q}$, respectively, where
	\begin{equation}\label{eq:pq_signals}
		\mbf{p}(t)=\mcl{C}_{p}\mbs{\psi}(t), \qquad	\mbf{q}(t)=\mcl{C}_{q}\mbf{v}(t) + \mcl{D}_{qp}\mbf{p}(t).
	\end{equation}
\end{prop}

\begin{proof}
	Let $\mbf{v}(t)\in\text{Z}^{\enn{1}}$ and $\mbs{\psi}(t)\in\text{Z}^{\enn{2}}$ for $t\geq 0$. 
	Then, $\mbf{p}(t)$ and $\mbf{q}(t)$ satisfy the PIEs defined by $\mbf{G}_{2}$ and $\mbf{G}_{1}$, respectively, if and only if they are as in~\eqref{eq:pq_signals}. In that case,
	\begin{align*}
		&\mcl{T}\slbmat{\mbf{v}_{t}(t)\\\mbs{\psi}_{t}(t)}\!-\!\mcl{A}\slbmat{\mbf{v}(t)\\\mbs{\psi}(t)}
		\!=\!\slbmat{\mcl{T}_{1}\mbf{v}_{t}(t)+\mcl{T}_{p}\mcl{C}_{p}\mbs{\psi}_{t}(t)\\\mcl{T}_{q}\mcl{C}_{q}\mbf{v}_{t}(t) +\mcl{T}_{2}\mbs{\psi}_{t}(t)+\mcl{T}_{q}\mcl{D}_{qp}\mcl{C}_{p}\mbs{\psi}_{t}(t)} \\
		&\hspace*{1.0cm} -\slbmat{\mcl{A}_{1}\mbf{v}(t)+\mcl{B}_{p}\mcl{C}_{p}\mbs{\psi}(t)\\\mcl{B}_{q}\mcl{C}_{q}\mbf{v}(t) +\mcl{A}_{2}\mbs{\psi}(t)+\mcl{B}_{q}\mcl{D}_{qp}\mcl{C}_{p}\mbs{\psi}(t)}	\\
		&\hspace*{2.25cm}=\slbmat{\mcl{T}_{p}\mbf{p}_{t}(t)+\mcl{T}_{1}\mbf{v}_{t}(t) -\mcl{A}_{1}\mbf{v}(t)-\mcl{B}_{p}\mbf{p}(t)\\\mcl{T}_{q}\mbf{q}_{t}(t) +\mcl{T}_{2}\mbs{\psi}_{t}(t) -\mcl{A}_{2}\mbs{\psi}(t)-\mcl{B}_{q}\mbf{q}(t)}.
	\end{align*}
	From this expression, it follows that $\smallbmat{\mbf{v}(t)\\\mbs{\psi}(t)}$ satisfies the PIE defined by $\mbf{G}$ if and only if $\mbf{v}(t)$ and $\mbs{\psi}(t)$ satisfy the PIEs defined by $\mbf{G}_{1}$ and $\mbf{G}_{2}$, respectively. 
\end{proof}

Using this result, we finally construct a PIE representation for the full delayed PDE in~\eqref{eq:PDDE_expanded_N2}.


\begin{cor}\label{cor:PIE_PDDE}
	Suppose that $A,A_{d}\in L_{\infty}^{n\times 3n}[\Omega_{a}^{b}]$, $\tau>0$ and $B\in\R^{2n\times 4n}$ satisfies the conditions of Lem.~\ref{lem:Tmap_1D}. Define $\mcl{T}$ and $\mcl{A}$ as in Block~\ref{block:Top_Aop}. Then, $\smallbmat{\mbf{v}\\\mbs{\psi}}$ is a solution to the PIE defined by $\{\mcl{T},\mcl{A}\}$ with initial state $\smallbmat{\mbf{v}_{0}\\\mbs{\psi}_{0}}$ if and only if $\smallbmat{\mbf{u}\\\mbs{\phi}}=\mcl{T}\smallbmat{\mbf{v}\\\mbs{\psi}}$ is a solution to the DPDE defined by $\{A,A_{d},B,\tau\}$ with initial state $\smallbmat{\mbf{u}_{0}\\\mbs{\phi}_{0}}=\mcl{T}\smallbmat{\mbf{v}_{0}\\\mbs{\psi}_{0}}$. Conversely, $\smallbmat{\mbf{u}\\\mbs{\phi}}$ is a solution to the DPDE defined by $\{A,A_{d},B,\tau\}$ with initial state $\smallbmat{\mbf{u}_{0}\\\mbs{\phi}_{0}}$ if and only if $\smallbmat{\mbf{v}\\\mbs{\psi}}=\smallbmat{\partial_{x}^2\mbf{u}\\\partial_{s}\partial_{x}^2\mbs{\phi}}$ is a solution to the PIE defined by $\{\mcl{T},\mcl{A}\}$ with initial state $\smallbmat{\mbf{v}_{0}\\\mbs{\psi}_{0}}=\smallbmat{\partial_{x}^2\mbf{u}_{0}\\\partial_{s}\partial_{x}^2\mbs{\phi}_{0}}$.
\end{cor}
\begin{proof}
	By definition of the operators $\mcl{T},\mcl{A}$, and invoking Prop.~\ref{prop:PIE_interconnection}, $\smallbmat{\mbf{v}\\\mbs{\psi}}$ is a solution to the PIE defined by $\{\mcl{T},\mcl{A}\}$ if and only if $\mbf{v}$ and $\mbs{\psi}$ are solutions to the PIEs~\eqref{eq:PIE_1D} and~\eqref{eq:PIE_2D}, respectively. By Lem.~\ref{lem:PIE_1D} and Lem.~\ref{lem:PIE_2D}, it follows that $\smallbmat{\mbf{v}\\\mbs{\psi}}$ is a solution to the PIE defined by $\{\mcl{T},\mcl{A}\}$ if and only if $\mbf{u}=\mcl{T}_{1}\mbf{v}$ and $\mbs{\phi}=\mcl{T}_{1}\mbf{v}+\mcl{T}_{2}\mbs{\psi}$ are solutions to the PDEs~\eqref{eq:PDE_1D} and~\eqref{eq:PDE_2D}, respectively. Taking the interconnection of these PDEs, we finally conclude that $\smallbmat{\mbf{v}\\\mbs{\psi}}$ is a solution to the PIE defined by $\{\mcl{T},\mcl{A}\}$ if and only if $\smallbmat{\mbf{u}\\\mbs{\phi}}=\mcl{T}\smallbmat{\mbf{v}\\\mbs{\psi}}=\smallbmat{\mcl{T}_{1}&0\\\mcl{T}_{1}&\mcl{T}_{2}}\smallbmat{\mbf{v}\\\mbs{\psi}}$ is a solution to the DPDE defined by $\{A,A_{d},B,\tau\}$.
	The converse result follows by similar reasoning.
\end{proof}

\section{Testing Stability in the PIE Representation}\label{sec:stability}

Having established a bijective map between the solution of the delayed PDE~\eqref{eq:PDDE_expanded_N2} and that of an associated PIE, we now show that the PIE representation can be used to formulate a convex optimization problem to verify stability of the delayed PDE. To derive such a result for a general PIE as in~\eqref{eq:standard_PIE}, recall the space $\text{Z}^{\enn{}}[\Omega]$ from~\eqref{eq:Zspace}.
We define
\begin{equation*}
	\ip{\mbf{u}}{\mbf{v}}_{\text{Z}}
	=u_{0}^T v_{0} + \ip{\mbf{u}_{1}}{\mbf{v}_{1}}_{L_2} + \ip{\mbf{u}_{2}}{\mbf{v}_{2}}_{L_2} + \ip{\mbf{u}_{3}}{\mbf{v}_{3}}_{L_2}.
\end{equation*}
for $\mbf{u}=(u_{0},\mbf{u}_{1},\mbf{u}_{2},\mbf{u}_{3})\in\text{Z}^{\enn{}}$, $\mbf{v}=(v_{0},\mbf{v}_{1},\mbf{v}_{2},\mbf{v}_{3})\in\text{Z}^{\enn{}}$.

\begin{thm}\label{thm:stability_as_LPI}
	Let $\{\mcl{T},\mcl{A}\}\in\mbf{\Pi}^{\enn{}\times\enn{}}$, and suppose that there exist constants $\epsilon,\alpha>0$ and a PI operator $\mcl{P}:\text{Z}^{\enn{}}\to\text{Z}^{\enn{}}$ such that $\mcl{P}=\mcl{P}^*$, $\mcl{P}\succeq \epsilon^2 I$, and
	\begin{align}\label{eq:stability_LPI}
		\mcl{A}^*\mcl{P}\mcl{T}+\mcl{T}^*\mcl{P}\mcl{A}\preceq -2\alpha\mcl{T}^*\mcl{P}\mcl{T}.
	\end{align}
	Then, any solution $\mbf{w}$ to the PIE defined by $\{\mcl{T},\mcl{A}\}$ satisfies
	\begin{align*}
		\|\mcl{T}\mbf{w}(t)\|_{\text{Z}}\leq \frac{\zeta}{\epsilon}\|\mcl{T}\mbf{w}(0)\|_{\text{Z}} e^{-\alpha t},\quad \text{where}~\zeta:=\sqrt{\|\mcl{P}\|_{\mcl{L}_{\text{Z}}}}.
	\end{align*}
\end{thm}
\smallskip
\begin{proof}	
	Consider the candidate Lyapunov functional $V(\mbf{w})=\ip{\mcl{T}\mbf{w}}{\mcl{P}\mcl{T}\mbf{w}}_{\text{Z}}$. Since $\mcl{P}\succeq\epsilon^2 I$ and $\|\mcl{P}\|_{\mcl{L}_{\text{Z}}}=\zeta^2$, this function is bounded above and below as
	\begin{equation*}
		\epsilon^2\|\mcl{T}\mbf{w}\|^2_{\text{Z}}\leq  V(\mbf{w})\leq \zeta^2\|\mcl{T}\mbf{w}\|_{\text{Z}}^2.
	\end{equation*}
	Now, let $\mbf{w}$ be an arbitrary solution to the PIE defined by $\{\mcl{T},\mcl{A}\}$. Then, the temporal derivative of $V$ along $\mbf{w}$ satisfies
	\begin{align*}
		\dot{V}(\mbf{w})
		&=\ip{\mcl{T}\mbf{w}_{t}}{\mcl{P}\mcl{T}\mbf{w}}_{\text{Z}}
		+\ip{\mcl{T}\mbf{w}}{\mcl{P}\mcl{T}\mbf{w}_{t}}_{\text{Z}}    \\
		&=\ip{\mcl{A}\mbf{w}}{\mcl{P}\mcl{T}\mbf{w}}_{\text{Z}}
		+\ip{\mcl{T}\mbf{w}}{\mcl{P}\mcl{A}\mbf{w}}_{\text{Z}}    \\
		&=\ip{\mbf{w}}{\left(\mcl{A}^*\mcl{P}\mcl{T} \!+\! \mcl{T}^*\mcl{P}\mcl{A}\right)\mbf{w}}_{\text{Z}}    \\
		&\leq -2\alpha\ip{\mcl{T}\mbf{w}}{\mcl{P}\mcl{T}\mbf{w}}_{\text{Z}}
		\hspace*{1.5cm}= -2\alpha V(\mbf{w}).
	\end{align*}
	Applying the Gr\"onwall-Bellman inequality, it follows that $V(\mbf{w}(t))\leq V(\mbf{w}(0))e^{-2\alpha t},$
	and therefore
	\begin{equation*}
		\|\mcl{T}\mbf{w}(t)\|^2_{\text{Z}}\leq (\zeta/\epsilon)^2\|\mcl{T}\mbf{w}(0)\|_{\text{Z}}^2 e^{-2\alpha t}.
	\end{equation*}
	\ \\[-2.6em]
\end{proof}

Thm.~\ref{thm:stability_as_LPI} shows that, for a PIE defined by $\{\mcl{T},\mcl{A}\}$, feasibility of the Linear PI Inequality (LPI)~\eqref{eq:stability_LPI} proves exponential stability of the function $\mcl{T}\mbf{w}$ for all solutions $\mbf{w}$ to the PIE.
By Cor.~\ref{cor:PIE_PDDE}, we also know that $\mbf{w}=\smallbmat{\mbf{v}\\\mbs{\psi}}$ is a solution to the PIE defined by $\{\mcl{T},\mcl{A}\}$ as in Block~\ref{block:Top_Aop} if and only if $\smallbmat{\mbf{u}\\\mbs{\phi}}=\mcl{T}\smallbmat{\mbf{v}\\\mbs{\psi}}=\mcl{T}\mbf{w}$ is a solution to the DPDE~\eqref{eq:PDDE_expanded_N2}, with $\smallbmat{\mbf{v}\\\mbs{\psi}}=\smallbmat{\mbf{u}_{xx}\\\mbs{\phi}_{sxx}}$. Using this result, we can thus test stability of solutions $\smallbmat{\mbf{u}\\\mbs{\phi}}$ to the DPDE as follows.

\begin{cor}\label{cor:stability_PDDE}
	Let $\{A,A_{d},B,\tau\}$ define a DPDE system, and let $\{\mcl{T},\mcl{A}\}$ define the associated PIE representation as in Cor.~\ref{cor:PIE_PDDE}. Suppose that there exist $\epsilon,\alpha,\zeta>0$ and $\mcl{P}$ satisfying the conditions of Thm.~\ref{thm:stability_as_LPI}. Then any solution $\smallbmat{\mbf{u}\\\mbs{\phi}}$ to the DPDE defined by $\{A,A_{d},B,\tau\}$ satisfies
	\begin{align*}
		\norm{\smallbmat{\mbf{u}(t)\\\mbs{\phi}(t)}}_{\text{Z}}\leq \frac{\zeta}{\epsilon}\norm{\smallbmat{\mbf{u}(0)\\\mbs{\phi}(0)}}_{\text{Z}} e^{-\alpha t}.
	\end{align*}
\end{cor}
\begin{proof}
	Let $\smallbmat{\mbf{u}\\\mbs{\phi}}$ be a solution to the DPDE defined by $\{A,A_{d},B,\tau\}$, and let $\smallbmat{\mbf{v}\\\mbs{\psi}}:=\smallbmat{\mbf{u}_{xx}\\\mbs{\phi}_{sxx}}$. By Cor.~\ref{cor:PIE_PDDE}, $\smallbmat{\mbf{v}\\\mbs{\psi}}$ solves the PIE defined by $\{\mcl{T},\mcl{A}\}$, and $\smallbmat{\mbf{u}\\\mbs{\phi}}=\mcl{T}\smallbmat{\mbf{v}\\\mbs{\psi}}$. By Thm.~\ref{thm:stability_as_LPI},
	\begin{align*}
		\norm{\smallbmat{\mbf{u}(t)\\\mbs{\phi}(t)}}_{\text{Z}}
		&=\norm{\mcl{T}\smallbmat{\mbf{v}(t)\\\mbs{\psi}(t)}}_{\text{Z}}	\\
		&\leq \frac{\zeta}{\epsilon}\norm{\mcl{T}\smallbmat{\mbf{v}(0)\\\mbs{\psi}(0)}}_{\text{Z}} e^{-\alpha t}=\frac{\zeta}{\epsilon}\norm{\smallbmat{\mbf{u}(0)\\\mbs{\phi}(0)}}_{\text{Z}} e^{-\alpha t}.\quad
	\end{align*}
	\ \\[-2.6em]
\end{proof}

By Cor.~\ref{cor:stability_PDDE}, we can finally test stability of the DPDE~\eqref{eq:PDDE_expanded_N2}, by the solving LPI~\eqref{eq:stability_LPI}. 
Note that stability is proven in the norm $\norm{\smallbmat{\mbf{u}(t)\\\mbs{\phi}(t)}}_{\text{Z}}^2=\|\mbf{u}(t)\|_{L_2}^2+\int_{0}^{1}\|\mbs{\phi}(t,s)\|_{L_2}^2ds$, bounding both the PDE state $\mbf{u}(t)$ and its history $\mbs{\phi}(t,s)=\mbf{u}(t-s\tau)$.

In order to numerically solve the LPI~\eqref{eq:stability_LPI}, we note that we can parameterize a cone of positive semidefinite PI operators $\mcl{P}\succeq 0$ by positive semidefinite matrices $P\succeq 0$ as
\begin{align*}
	\Pi_{+}:=\{\mcl{P}\in\Pi \mid \mcl{P}=\mcl{Z}^* \text{M}_{P} \mcl{Z},\ \mcl{Z}\in\Pi,\ P\succeq 0\},
\end{align*}
where $\mcl{Z}$ is some fixed PI operator. 
Then, if $\mcl{P}\in\Pi_{+}$, there exists some $P=[P^{1/2}]^T P^{1/2}\succeq 0$ and $\mcl{Z}\in\Pi$ such that, for any $\mbf{u}$ in the domain of $\mcl{P}$
\begin{align*}
	\ip{\mbf{u}}{\mcl{P}\mbf{u}}=\ip{\text{M}_{P^{1/2}}\mcl{Z}\mbf{u}}{\text{M}_{P^{1/2}}\mcl{Z}\mbf{u}}\geq 0,
\end{align*}
guaranteeing that $\mcl{P}\succeq 0$. In this manner, LPI conditions such as $\mcl{P}\succeq 0$ can be posed as LMI constraints $P\succeq 0$, allowing feasibility of the LPI in Thm.~\ref{thm:stability_as_LPI} to be numerically tested using semidefinite programming. This process of parsing LPIs as LMIs has been automated in the MATLAB software package PIETOOLS~\cite{shivakumar2021PIETOOLS}, wherein the operator $\mcl{Z}=\mcl{Z}_{d}$ used in the parameterization of $\mcl{P}\succeq 0$ is defined by monomials of degree at most $d$ in each of the spatial variables. Conservatism introduced through the parameterization $\mcl{P}=\mcl{Z}_{d}^* \text{M}_{P}\mcl{Z}_{d}$ can then be reduced by increasing the maximal degree $d$ of the monomials, at the expense of increasing the size of the matrix $P\succeq 0$ and thus the complexity of the LMI.

In Section~\ref{sec:numerical_examples}, we use the PIETOOLS software to test stability of several delayed PDE systems.

\section{PDEs with Delay at Boundary}\label{sec:other_delay}

In the previous section, we showed how stability of a 2nd order, linear, 1D PDE with delay can be verified by solving a LPI, based on the PIE representation associated to this PDE. This result can be readily generalized to coupled systems of $N$th order, linear, 1D PDEs with delay, using the fundamental state $\mbf{v}(t)=\partial_{x}^{N}\mbf{u}(t)$ for PDE state $\mbf{u}(t)\in H_{N}$, and using the formulae from~\cite{shivakumar2022GPDE_Arxiv} to compute an associated operator $\mcl{T}_{1}$ such that $\mcl{T}_{1}\mbf{v}(t)=\mbf{u}(t)$. 

More generally, using the formulae presented in~\cite{shivakumar2022GPDE_Arxiv}, a PIE representation for a very general class of PDEs with delay can be constructed. For example, consider a 1D PDE coupled to an ODE with delay, taking the form
\begin{align}\label{eq:PDE_ODEdelay}
	&\dot{w}(t)=A_{00}w(t) +A_{00,d}w(t-\tau) +A_{01}\smallbmat{\mbf{u}(t,0)\\\mbf{u}_{x}(t,0)\\\mbf{u}(t,1)\\\mbf{u}_{x}(t,1)},	\\
	&\mbf{u}_{t}(t)=A_{10}w(t)+A_{11}\slbmat{\mbf{u}(t)\\\mbf{u}_{x}(t)\\ \mbf{u}_{xx}(t)},\hspace*{1.2cm} \mbf{u}(t)\in H_{2}^{n}[\Omega_{0}^{1}],	\notag\\
	&B_{0}w(t) +B_{1}\smallbmat{\mbf{u}(t,0)\\\mbf{u}_{x}(t,0)\\\mbf{u}(t,1)\\\mbf{u}_{x}(t,1)}
	= 0 \hspace*{3.5cm} t\geq 0,	\notag
\end{align}
Introducing the state $\phi(t,x)=w(t-x\tau)$ for $x\in[0,1]$, this system can be equivalently represented as
\begin{align*}
	&\dot{w}(t)=A_{00}w(t) +A_{00,d}\phi(t,1) +A_{01}\smallbmat{\mbf{u}(t,0)\\\mbf{u}_{x}(t,0)\\\mbf{u}(t,1)\\\mbf{u}_{x}(t,1)},	\\
	&\phi_{t}(t)=-(1/\tau)\phi_{x}(t),\hspace*{3.25cm} \mbf{v}(t)\in H_{1}[\Omega_{0}^{1}],	\\
	&\mbf{u}_{t}(t)=A_{10}w(t)+A_{11}\slbmat{\mbf{u}(t)\\\mbf{u}_{x}(t)\\ \mbf{u}_{xx}(t)},\hspace*{1.5cm} \mbf{u}(t)\in H_{2}^{n}[\Omega_{0}^{1}],	\notag\\
	&B_{0}w(t) +B_{1}\smallbmat{\mbf{u}(t,0)\\\mbf{u}_{x}(t,0)\\\mbf{u}(t,1)\\\mbf{u}_{x}(t,1)}
	= 0 \qquad w(t)-\mbf{v}(t,0)=0,\quad t\geq 0.	\notag
\end{align*}
Then, the dynamics are defined by a coupled ODE - 1D PDE system, on the state $\smallbmat{w(t)\\\mbs{\phi}(t)\\\mbf{u}(t)}$. Using the formulae from~\cite{shivakumar2022GPDE_Arxiv}, an equivalent PIE representation for this system can be readily constructed, introducing the fundamental state $\smallbmat{w(t)\\\mbs{\phi}_{x}(t)\\\mbf{u}_{xx}(t)}$.

Similarly, consider a 1D PDE with delay in the boundary conditions,
\begin{align}\label{eq:PDE_BCdelay}
	&\mbf{u}_{t}(t)=A\slbmat{\mbf{u}(t)\\\mbf{u}_{x}(t)\\ \mbf{u}_{xx}(t)},\hspace*{2.0cm} \mbf{u}(t)\in H_{2}^{n}[\Omega_{0}^{1}],	\\
	&B\smallbmat{\mbf{u}(t,0)\\\mbf{u}_{x}(t,0)\\\mbf{u}(t,1)\\\mbf{u}_{x}(t,1)} +B_{d}\smallbmat{\mbf{u}(t-\tau,0)\\\mbf{u}_{x}(t-\tau,0)\\\mbf{u}(t-\tau,1)\\\mbf{u}_{x}(t-\tau,1)}
	= 0 \hspace*{1.5cm} t\geq 0,\notag
\end{align}
In this case, introducing $\mbs{\phi}(t,s)=\smallbmat{\mbf{u}(t-s\tau,0)\\\mbf{u}_{x}(t-s\tau,0)\\\mbf{u}(t-s\tau,1)\\\mbf{u}_{x}(t-s\tau,1)}$ for $s\in[0,1]$, the PDE may be represented as
\begin{align*}
	&\mbf{u}_{t}(t)=A(x)\slbmat{\mbf{u}(t)\\\mbf{u}_{x}(t)\\ \mbf{u}_{xx}(t)},\hspace*{2.5cm} \mbf{u}(t)\in H_{2}^{n}[\Omega_{0}^{1}],	\\
	&\mbs{\phi}_{t}(t)=-(1/\tau)\mbs{\phi}_{s}(t),\hspace*{2.5cm} \mbs{\phi}(t)\in H_{1}^{4n}[\Omega_{0}^{1}],	\\
	&B\smallbmat{\mbf{u}(t,0)\\\mbf{u}_{x}(t,0)\\\mbf{u}(t,1)\\\mbf{u}_{x}(t,1)} +B_{d}\mbs{\phi}(t,1)= 0,	\qquad
	\mbs{\phi}(t,0)-\smallbmat{\mbf{u}(t,0)\\\mbf{u}_{x}(t,0)\\\mbf{u}(t,1)\\\mbf{u}_{x}(t,1)}=0.
\end{align*}
This representation is again expressed only in terms of coupled 1D PDEs on the state $\smallbmat{\mbf{u}(t)\\\mbs{\phi}(t)}$, involving no explicit delay. Introducing the associated fundamental state $\smallbmat{\mbf{u}_{xx}(t)\\\mbs{\phi}_{x}(t)}$, this system too can be converted to an equivalent PIE using the formulae from~\cite{shivakumar2022GPDE_Arxiv}.

Using this approach of modeling the delay using a transport equation, a very general class of infinite-dimensional systems with delay can be represented as coupled ODE-PDE systems, with the PDE being either 2D or 1D depending on whether the delay occurs in the dynamics (as in the previous section) or in the boundary conditions (as in Eqn.~\eqref{eq:PDE_BCdelay}). Then, an equivalent PIE representation of the system can be constructed using the methodology proposed in the previous section or that presented in~\cite{shivakumar2022GPDE_Arxiv}, at which point stability can be readily tested by solving the LPI from Thm.~\ref{thm:stability_as_LPI}. The appendices of this paper provide a full, though somewhat abstract overview of how a PIE representation can be constructed for a general class of ODE-PDE systems with delay. However, this process of computing the PIE representation and subsequently testing stability has also already been automated in the Matlab software suite PIETOOLS 2022~\cite{shivakumar2021PIETOOLS}, applying semidefinite programming to test feasibility of the LPI~\eqref{eq:stability_LPI}. In the next section, we use this software to numerically test stability of several PDE systems with delay.


%

\section{Numerical Examples}\label{sec:numerical_examples}

In this section, we provide several numerical examples, illustrating how stability of different ODE-PDE systems with delay can be numerically tested by verifying feasibility of the LPI from Thm.~\ref{thm:stability_as_LPI}. In each case, the PIETOOLS software package~\cite{shivakumar2020PIE_pietools} is used to declare the delayed system as a coupled systems of ODEs and PDEs, convert the system to an equivalent PIE, and subsequently declare and solve the stability LPI.

\subsection{Heat Equation with Delay in PDE}

Consider a heat equation with a delayed reaction term, as studied in~\cite{fridman2009stabilityDPDE,ccalicskan2009stability}
\begin{align}\label{eq:Example2}
	\mbf{u}_{t}(t,x)&=\mbf{u}_{xx}(t,x) + r\mbf{u}(t,x) - \mbf{u}(t-\tau,x),\quad x\in\Omega_{0}^{\pi},	\notag\\
	\mbf{u}(t,0)&=\mbf{u}(t,\pi)=0.
\end{align}
We can model the delay using a 2D transport equation as
\begin{align*}
	\mbf{u}_{t}(t,x)&=\mbf{u}_{xx}(t,x) + r\mbf{u}(t,x) - \mbf{v}(t,1,x),	& &\hspace*{-1.5cm}x\in\Omega_{0}^{\pi},	\\
	\mbs{\phi}_{t}(t,s,x)&=-(1/\tau)\mbs{\phi}_{s}(t,s,x),& &\hspace*{-1.4cm}s\in\Omega_{0}^{1}, \\
	\mbf{u}(t,0)&=\mbf{u}(t,\pi)=0, \qquad
	\mbs{\phi}(t,s,0)=\mbs{\phi}(t,s,\pi)=0, \\
	\mbs{\phi}(t,0,x)&=\mbf{u}(t,x).
\end{align*}
Using PIETOOLS, we then obtain a PIE representation
\begin{align*}
	&(\mcl{T}\mbf{v}_{t})(t,x)
	=\!\mbf{v}(t,x)\!+\!(r\!-\!1)(\mcl{T}\mbf{v})(t,x) \!-\!\! \int_{0}^{1}\!\!(\mcl{T}\mbs{\psi})(t,s,x)ds, \\[-0.2em]
	&(\mcl{T}\mbf{v}_{t})(t,x)+\int_{0}^{s}(\mcl{T}\mbs{\psi}_{t})(t,\nu,x)d\nu =-\frac{1}{\tau}(\mcl{T}\mbs{\psi})(t,s,x),
\end{align*}
where $\mbf{v}(t,x)=\partial_{x}^2\mbf{u}(t,x)$, $\mbs{\psi}(t,s,x)=\partial_{s}\partial_{x}^2\mbs{\phi}(t,s,x)$, and
\begin{align*}
	\bl(\mcl{T}\mbf{v}\br)(t,x)=\int_{0}^{x}\!\theta(x-1)\mbf{v}(t,\theta)d\theta + \int_{x}^{\pi}\!\! x(\theta-1)\mbf{v}(t,\theta)d\theta.
\end{align*}
In~\cite{ccalicskan2009stability}, it was shown that for $0<r<2$, the DPDE~\eqref{eq:Example2} is stable if and only if and only if $\tau<\bar{\tau}:=\frac{\cos^{-1}(r-1)}{\sqrt{2r-r^2}}$. Performing bisection on the value of the delay $\tau$, stability for different values of $r$ can be numerically verified with PIETOOLS for delays up to $\tau_{\text{LPI}}$ as presented in Tab.~\ref{tab:tau}. Here, for each test, the LPI~\eqref{eq:stability_LPI} was numerically parsed as an LMI by parameterizing the positive operator $\mcl{P}\succeq 0$ by a symmetric positive semidefinit matrix $P\in\R^{27\times 27}$. The associated Lyapunov-Krasovskii functional is then parameterized by $\frac{1}{2}27(27+1)=378$ decision variables, substantially more complicated than the Lyapunov-Krasovskii functional used to test stability in~\cite{fridman2009stabilityDPDE}, only involving 5 decision variables. Although the resulting stable delay bound is much less conservative using the LPI approach ($\tau_{\text{LPI}}=1.03470$ versus $\tau=1.025$ for $r=1.9$), the computational complexity of the LMI used to achieve this bound is also much greater.

%
 
 \begin{table}[h!]
 		\setlength{\tabcolsep}{5pt}
 		\begin{tabular}{l|cccccc}
 			$r$ 	 			& 0.1 	 & 0.5 & 0.8 & 1.2 & 1.5 & 1.9	\\\hline
 			$\bar{\tau}$ 		& 6.17258 & 2.41839 & 1.80870 & 1.39768 & 1.20920 & 1.03472	\\
 			$\tau_{\text{LPI}}$ & 6.17248 & 2.41837 & 1.80869 & 1.39767 & 1.20919 & 1.03470
 		\end{tabular}	
 		\caption{
 			Maximal delay $\tau_{\text{LPI}}$ for which exponential stability of System~\eqref{eq:Example2} was verified using Thm.~\ref{thm:stability_as_LPI} with $\epsilon=10^{-2}$, $\alpha=0$.
 		}\label{tab:tau}
 	\vspace*{-0.5cm}
 \end{table}


%

%

\subsection{Wave Equation with Delay in Boundary}

Consider a wave equation with delay in the boundary,
\begin{align}\label{eq:Example3}
	\mbf{u}_{tt}(t,x)&=\partial_{x}^2 \mbf{u}(t,x)	\hspace*{3.5cm}	x\in\Omega_{0}^{1}, \\
	\mbf{u}(t,0)&=0,	\hspace*{0,5cm}
	\partial_{x}\mbf{u}(t,1)=(1-\mu)\mbf{u}_{t}(t,1) + \mu \mbf{u}_{t}(t-\tau,1).	\notag
\end{align}
where $\mu\in (0,1)$. Introducing
\begin{align*}
	&\mbf{u}_{1}(t)=\mbf{u}(t),	\hspace*{3.0cm}
	\mbf{u}_{2}(t)=\mbf{u}_{t}(t),	\\
	&\mbs{\phi}_{1}(t,x)=\mbf{u}_{1}(t-\tau x,1), \hspace*{1.25cm}
	\mbs{\phi}_{2}(t,x)=\mbf{u}_{2}(t-\tau x,1),	\\
	&u_{0}(t)=(1-\mu)\mbf{u}_{1}(t,1)+\mu\mbs{\phi}_{1}(t,1)
\end{align*}
this system can be equivalently represented as
\begin{align*}
	\dot{u}_{0}(t)&=\partial_{x}\mbf{u}_{1}(t,1) \\
	\partial_{t}\mbf{u}_{1}(t,x)&=\mbf{u}_{2}(t,x),	\hspace*{1.5cm}
	\partial_{t}\mbf{u}_{2}(t,x)=\partial_{x}^2\mbf{u}_{1}(t,x),	\\
	\partial_{t}\mbs{\phi}_{1}(t,x)&=-\frac{1}{\tau}\partial_{x}\mbs{\phi}_{1}(t,x),	\hspace*{0.6cm}
	\partial_{t}\mbs{\phi}_{2}(t,x)=-\frac{1}{\tau}\partial_{x}\mbs{\phi}_{2}(t,x),	\\
	\mbf{u}_{1}(t,0)&=0,	\hspace*{2.5cm} \mbf{u}_{2}(t,0)=0,	\\
	\mbs{\phi}_{1}(t,0)&=\mbf{u}_{1}(t,1), \hspace*{1.55cm}
	\mbs{\phi}_{2}(t,0)=\mbf{u}_{2}(t,1), \\
	u_{0}(t)&=(1-\mu)\mbf{u}_{1}(t,1)+\mu\mbs{\phi}_{1}(t,1), \\
	\partial_{x}\mbf{u}_{1}(t,1)&=(1-\mu)\mbf{u}_{2}(t,1)+\mu\mbs{\phi}_{2}(t,1).	
\end{align*}
This ODE-PDE system can be readily declared in PIETOOLS, and converted to a PIE.

In~\cite{xu2006Boundary_stabilization_wave_eq}, the PDE~\eqref{eq:Example3} was proven to be stable independent of delay if $\mu<\frac{1}{2}$, and unstable independent of delay if $\mu>\frac{1}{2}$. We examine the ability of the proposed algorithm to expand upon this result by determining bounds on the rate of decay for several values of $\mu$ and $\tau$. First, fixing $\tau=1$, we note that stability can be numerically verified with PIETOOLS for any $\mu\leq 0.5-10^{-3}$. Next, fixing $\mu=0.4$ and performing bisection on the value of $\alpha$, exponential decay rates can be computed as illustrated in Tab.~\ref{tab:tau}. For each test, the operator $\mcl{P}$ in the LPI~\eqref{eq:stability_LPI} was parameterized by a symmetric positive semidefinite matrix $P\in\R^{73\times 73}$.
\begin{table}[h!]
	\setlength{\tabcolsep}{5pt}
		\begin{tabular}{l|ccccccc}
			$\tau$ 	 & 0.125    &  0.25   & 0.5    & 1.0    & 2.0    & 4.0 & 8.0	\\\hline
			$\alpha$ & 0.2023 & 0.1908  & 0.1513  & 0.1333 & 0.1060  & 0.0701   & 0.0135 \\[-0.6em]
		\end{tabular}
		\caption{
			Decay rates $\alpha$ for which exponential stability of System~\eqref{eq:Example3}\\[-0.1em] 
			with $\mu=0.4$ was verified using Thm.~\ref{thm:stability_as_LPI} with $\epsilon=10^{-3}$.
		}\label{tab:alpha}
	\vspace*{-0.5cm}
\end{table}

\section{Conclusion}

In this paper, an LMI-based method for verifying stability of coupled, linear, delayed, PDE systems in a single spatial dimension was presented. In particular, it was shown that for any suitably well-posed PDE with delay, there exists an associated (1D or 2D) PIE with a corresponding bijective map from solution of the delayed PDE to that of the PIE. The PIE representation was then used to propose a stability test for the delayed PDE. This stability test was posed as a linear operator inequality expressed using PI operator variables (an LPI). Finally, the PIETOOLS software package was used to convert the LPI to a semidefinite programming problem and the resulting stability conditions were applied to several common examples of delayed PDEs. While these results only apply to fixed, constant delays, an extension to time-varying delays may be possible using PDE representations such as in~\cite{nicaise2009stability_waveeq_heateq_BCdelay}.

	
	\vspace*{-0.1cm}

	\bibliographystyle{IEEEtran}
	\bibliography{bibfile}
	
\clearpage
	
\begin{appendices}

\section{Feedback Interconnection of Partial Integral Equations}\label{appx:PIE_interconnection}

In order to derive the PIE representation of a PDE with delay, in Section~\eqref{appx:PIEmap}, we will first represent this delayed PDE as the feedback interconnection of a 1D PDE and a 2D PDE. Separately deriving a PIE representation of both the 1D PDE and 2D PDE subsystems, we can then construct a PIE representation of the original system by taking the feedback interconnection of the two PIE subsystems.

In this section, we prove that the feedback interconnection of PIEs can indeed be represented as a PIE as well. This result was already proven in~\cite{shivakumar2022GPDE_Arxiv}, for the case of finite-dimensional interconnection signals, and will be extended here to include infinite-dimensional interconnection signals. To start, in the following subsection, we first suitable classes of PI operators to parameterize our PIEs. Although more general classes of PI operators have already been defined in other papers, the classes in the following subsection are sufficient for the purposes of this paper. In the next subsection, we then show how the interconnection of two suitable PIEs can be represented as a PIE as well.

\subsection{Algebras of PI Operators in 2D}

Partial integral (PI) operators are bounded, linear operators, parameterized by square integrable functions. In 1D, the standard class of PI operators is that of 3-PI operators, which we define as follows
\begin{defn}[3-PI Operators ($\Pi_{3}$)]
	For $m,n\in\N$, define
	{\begin{align*}
			\mcl{N}_{3}^{m\times n}[\Omega_{a}^{b}]\!:=\!L_{\infty}^{m\times n}[\Omega_{a}^{b}] \!\times\! L_2^{m\times n}[\Omega_{a}^{b}\!\times\!\Omega_{a}^{b}] \times L_2^{m\times n}[\Omega_{a}^{b}\!\times\!\Omega_{a}^{b}].
	\end{align*}}%
	Then, for given parameters $N:=\{N_0,N_1,N_2\}$, we define the associated 3-PI operator for $\mbf{u}\in L_2^{n}[\Omega_{a}^{b}]$ as
	{\begin{align*}
			\small
			\bl(\mcl{P}[N]\mbf{u}\br)(x)&= N_0(x)\mbf{u}(x) + \int_{a}^{x}\! N_1(x,\theta)\mbf{u}(\theta)d\theta 	\\[-0.4em]
			&\qquad+ \int_{x}^{b}\! N_2(x,\theta)\mbf{u}(\theta)d\theta, \qquad x\in \Omega_{a}^{b}.\\[-1.4em]
	\end{align*}}%
\end{defn}
Since we are interested in coupled ODE-PDE systems, we need to be able to map between finite-dimensional ODE states $u\in\R^{n_{0}}$ and infinite-dimensional PDE states $\mbf{u}_{1}\in L_2^{n_{1}}[\Omega_{a}^{b}]$. For this purpose, we define a class of 4-PI operators, acting on the function space $\text{Z}_{1}^{\enn{}}[\Omega_{a}^{b}]:=\R^{n_{0}}\times L_2^{n_{1}}[\Omega_{a}^{b}]$ for $\enn{}=(n_{1},n_{1})\in\N^{2}$.

\begin{defn}[4-PI Operators ($\Pi_{4}$)]
	For given\\ $\text{m}:=(m_0,m_1)\in\N^{2}$, $\text{n}:=(n_0,n_1)\in\N^2$, define
	\begin{align*}
		\mcl{N}_{4}^{\text{m}\times\text{n}}[\Omega_{a}^{b}]:=\left[\!
		\begin{array}{ll}
			\R^{m_0\times n_0}&L_2^{m_0\times n_1}[\Omega_{a}^{b}]\\
			L_2^{m_1\times n_0}&\mcl{N}_3^{m_1\times n_1}[\Omega_{a}^{b}].
		\end{array}\!
		\right].
	\end{align*}
	Then, for given parameters $B=\smallbmat{B_{00}&B_{01}\\B_{10}&B_{11}}\in\mcl{N}_{4}^{\text{m}\times\text{n}}[\Omega_{a}^{b}]$, we define the associated 4-PI operator for $\mbf{u}=\srbmat{u_0\\\mbf{u}_1}\in\srbmat{\R^{n_0}\\L_2^{n_1}}$ as
	\begin{align*}
		\bl(\mcl{P}[B]\mbf{u}\br)(x)&=\!\left[\!
		\begin{array}{rcl}
			B_{00}u_0 &\!+\!\! &\int_{a}^{b}B_{01}(x)\mbf{u}_1(x)dx \\
			B_{10}(x)u_0 &\!+\!\!& \bl(\mcl{P}[B_{11}]\mbf{u}_1)(x)
		\end{array}\!
		\right],	&	\!x&\in\Omega_{a}^{b}.
	\end{align*}
\end{defn}

Finally, since our delayed system will actually yield 2D PDEs, we will need operators acting on 2D function spaces as well. In particular, we define the space $\text{Z}_{2}^{\text{n}}[\Omega_{ac}^{bd}]=L_2^{n_0}[\Omega_{c}^{d}]\times L_2^{n_1}[\Omega_{ac}^{bd}]$ for $\enn{}=(n_{0},n_{1})\in\N^{2}$, and for $\text{n}=(\text{n}_1,\text{n}_2)\in\N^4$, we define $\text{Z}_{12}^{\text{n}}[\Omega_{ac}^{bd}]:=\text{Z}_{1}^{\text{n}_1}[\Omega_{a}^{b}]\!\times \text{Z}_{2}^{\text{n}_{2}}[\Omega_{ac}^{bd}]$. Through some abuse of notation, we will allow 4-PI operators to act on functions in $\text{Z}_{2}:=\smallbmat{L_2[\Omega_{c}^{d}]\\L_2[\Omega_{ac}^{bd}]}$ as well, assuming them to act as multipliers along $y\in\Omega_{c}^{d}$. Then, we define a restricted class of 2D PI operators as follows.

\begin{defn}[PI Operators on 2D ($\Pi_{12}$)]
	For given $\text{m}:=(\text{m}_1,\text{m}_2)\in\N^4$ and $\text{n}:=(\text{n}_1,\text{n}_2)\in\N^4$, define
	\begin{align*}
		\mcl{N}_{12}^{\text{m}\times\text{n}}[\Omega_{ac}^{bd}]:=\left[\!
		\begin{array}{ll}
			\mcl{N}_{4}^{\text{m}_{1}\times\text{n}_{1}}[\Omega_{a}^{b}] & \mcl{N}_{4}^{\text{m}_{1}\times\text{n}_{2}}[\Omega_{a}^{b}] \\
			\mcl{N}_{4}^{\text{m}_{2}\times\text{n}_{1}}[\Omega_{a}^{b}] & (\mcl{N}_{4}^{\text{m}_{2}\times\text{n}_{2}}[\Omega_{a}^{b}])^{3}
		\end{array}\!
		\right].
	\end{align*}
	Then, for given parameters $B=\smallbmat{B_{11}&B_{12}\\B_{21}&N}\in\mcl{N}_{12}^{\text{m}\times\text{n}}[\Omega_{ac}^{bd}]$, define the associated PI operator for $\mbf{u}=\srbmat{\mbf{u}_1\\\mbf{u}_2}\in\srbmat{\text{Z}_{1}^{\enn{1}}[\Omega_{a}^{b}]\\\text{Z}_{2}^{\enn{2}}[\Omega_{ac}^{bd}]}$ and $(x,y)\in\Omega_{ac}^{bd}$ as
	\begin{align*}
		\bl(\mcl{P}[B]\mbf{u}\br)(x,y)&\!=\!\left[\!\!{\small
			\begin{array}{rl}
				\bl(\mcl{P}[B_{11}]\mbf{u}_1\br)(x) &\!+ \int_{c}^{d}\bl(\mcl{P}[B_{12}]\mbf{u}_2\br)(x,y)dy \\
				\bl(\mcl{P}[B_{21}]\mbf{u}_1\br)(x,y) &\!+ \bl(\mcl{P}[N]\mbf{u}_{2}\br)(x,y)
		\end{array}}\!
		\right]\!,
	\end{align*}
	where for $N=\{N_0,N_1,N_2\}\in(\mcl{N}_{4}^{\text{m}_2\times\text{n}_2}[\Omega_{a}^{b}])^3$,
	{\small\begin{align*}
			\bl(\mcl{P}[N]\mbf{u}_{2}\br)(x,y)
			&=\bl(\mcl{P}[N_{0}]\mbf{u}_{2}\br)(x,y) + \int_{c}^{y}\bl(\mcl{P}[N_{1}]\mbf{u}_{2}\br)(x,\nu)d\nu \\[-0.4em]
			&\qquad+ \int_{y}^{d}\bl(\mcl{P}[N_{2}]\mbf{u}_{2}\br)(x,\nu)d\nu.
	\end{align*}}%
\end{defn}

Throughout this paper, we will use $\Pi$ to denote the general class of PI operators, writing $\mcl{B}\in\Pi^{\text{m}\times\text{n}}$ if there exist parameters $B\in\mcl{N}_{12}^{\text{m}\times\text{n}}$ such that $\mcl{B}=\mcl{P}[B]$. We will make extensive use of the following properties of PI operators.

\begin{enumerate}
	\item
	The sum $\mcl{Q}+\mcl{R}=\mcl{P}\in\Pi^{\text{m}\times\text{n}}$ of two PI operators $\mcl{Q},\mcl{R}\in\Pi^{\text{m}\times\text{n}}$ is a PI operator. 
	
	\item
	The composition $\mcl{Q}\circ\mcl{R}=\mcl{P}\in\Pi^{\text{m}\times\text{n}}$ of two PI operators $\mcl{Q}\in \Pi^{\text{m}\times\text{p}}$, $\mcl{R}\in \Pi^{\text{p}\times\text{n}}$ is a PI operator. 
	
	%
	
	\item
	The adjoint $\mcl{P}^*\in\Pi^{\text{n}\times\text{m}}$ of a PI operator $\mcl{P}\in\Pi^{\text{m}\times\text{n}}$ is a PI operator. 
	
\end{enumerate}

We refer to~\cite{shivakumar2022GPDE_Arxiv} (1D) and~\cite{jagt2021PIEArxiv} (2D) for more details on PI operators, including explicit definitions of the parameters associated to operations such as addition and multiplication.



\subsection{A Feedback Interconnection of PIEs}

Having defined a sufficiently general class of PI operators for the purposes of this paper, consider now a PIE of the form
\begin{align}\label{eq:standard_PIE_full}
	\srbmat{\mcl{T}_{r}\mbf{r}_{t}(t)+\mcl{T}_{w}\mbf{w}_{t}(t)+\mcl{T}\mbf{v}_{t}(t)\\ \mbf{z}(t)\\\mbf{q}(t)}\!=\!
	\left[\! \arraycolsep=2.5pt
	\begin{array}{lll}
		\mcl{A}&\mcl{B}_{w}&\mcl{B}_{r} \\
		\mcl{C}_{z}&\mcl{D}_{zw}&\mcl{D}_{zr} \\
		\mcl{C}_{q}&\mcl{D}_{qw}&\mcl{D}_{qr} \\
	\end{array}\!
	\right]\!
	\srbmat{\mbf{v}(t)\\\mbf{w}(t)\\\mbf{r}(t)},
\end{align}
where at each time $t\geq 0$, $\mbf{v}(t)\in\text{Z}_{12}^{\enn{u}}$, $\mbf{w}(t)\in \text{Z}_{12}^{\enn{w}}$, $\mbf{r}(t)\in \text{Z}_{12}^{\enn{r}}$, $\mbf{z}(t)\in \text{Z}_{12}^{\enn{z}}$ and $\mbf{q}(t)\in \text{Z}_{12}^{\enn{q}}$, for some $\enn{u},\enn{w},\enn{z},\enn{r},\enn{q}\in\N^4$. We collect the PI operators defining the PIE in
\begin{align}\label{eq:standard_PIE_full_ops}
	\mbf{G}_{\text{pie}}\!=\!\!
	\left[\!{\small \arraycolsep = 2.4pt
		\begin{array}{lll}
			\mcl{T}&\mcl{T}_w &\mcl{T}_{r} \\
			\mcl{A}&\mcl{B}_{w} & \mcl{B}_{r} \\
			\mcl{C}_{z} &\mcl{D}_{zw} & \mcl{D}_{zr} \\
			\mcl{C}_{q} &\mcl{D}_{qw} & \mcl{D}_{qr}
	\end{array}}\!
	\right]
	\in
	\left[\!{\small
		\begin{array}{lll}
			\Pi^{\enn{u}\times\enn{u}} & \Pi^{\enn{u}\times\enn{w}} & \Pi^{\enn{u}\times\enn{r}} \\
			\Pi^{\enn{u}\times\enn{u}} & \Pi^{\enn{u}\times\enn{w}} & \Pi^{\enn{u}\times\enn{r}}\\
			\Pi^{\enn{z}\times\enn{u}} & \Pi^{\enn{z}\times\enn{w}} & \Pi^{\enn{z}\times\enn{r}}	\\
			\Pi^{\enn{q}\times\enn{u}} & \Pi^{\enn{q}\times\enn{w}} & \Pi^{\enn{q}\times\enn{r}}
	\end{array}}\!
	\right].
\end{align}
If the PIE involves only a single output and a single output signal, so that $\enn{q}=\enn{r}=0$, we will exclude the PI operators associated to $\mbf{q},\mbf{r}$, writing $\mbf{G}_{\text{pie}}=\{\mcl{T},\mcl{T}_{w},\mcl{A},\mcl{B}_{w},\mcl{C}_{z},\mcl{D}_{zw}\}$. If the PIE describes an autonomous system, so that also $\enn{w}=\enn{z}=0$, we will simply write $\mbf{G}_{\text{pie}}=\{\mcl{T},\mcl{A}\}$.

\begin{defn}[Solution to the PIE]
	For given input signals $(\mbf{w},\mbf{r})$ and given initial conditions $\mbf{v}_{0}\in \text{Z}_{12}^{\enn{u}}$, we say that $(\mbf{v},\mbf{z},\mbf{q})$ is a solution to the PIE defined by $\mbf{G}_{\text{pie}}$ if $\mbf{v}$ is Frech\'et differentiable, $\mbf{v}(0)=\mbf{v}_{0}$, and for all $t\geq0$, $(\mbf{v}(t),(\mbf{z}(t),\mbf{q}(t)),(\mbf{w}(t),\mbf{r}(t)))$ satisfies Eqn.~\eqref{eq:standard_PIE_full} with operators defined as in~\eqref{eq:standard_PIE_full_ops}.
\end{defn}

Using the composition and addition rules of PI operators, it is easy to show that the interconnection of two suitable PIEs can also be represented as a PIE.

\begin{prop}[Interconnection of PIEs]\label{prop:PIE_interconnection_full}
	Let 
	\begin{align*}	
		\mbf{G}_{\text{pie},1}&=
		\left[\!{\small
			\arraycolsep=2.5pt
			\begin{array}{lll}
				\mcl{T}_{1} & \mcl{T}_{1w} & \mcl{T}_{1r}\\
				\mcl{A}_{1} & \mcl{B}_{1w} & \mcl{B}_{1r} \\
				\mcl{C}_{z1} & \mcl{D}_{zw} & \mcl{D}_{zr} \\
				\mcl{C}_{q1} & \mcl{D}_{qw} & 0
		\end{array}}\!
		\right],
		\quad\text{and}\quad
		\mbf{G}_{\text{pie},2}=
		\left[\!{\small
			\arraycolsep=2.5pt
			\begin{array}{ll}
				\mcl{T}_{2} & \mcl{T}_{2q} \\
				\mcl{A}_{2} & \mcl{B}_{2q} \\
				\mcl{C}_{r2} & \mcl{D}_{rq}
		\end{array}}\!
		\right],
	\end{align*}
	define two PIEs. Define the associated PIE interconnection as
	$\mbf{G}_{\text{pie}}\!=\!\{\mcl{T},\mcl{T}_{w},\mcl{A},\mcl{B}_{w},\mcl{C}_{z},\mcl{D}_{zw}\}\!=\!\mcl{L}_{\text{pie}\times\text{pie}}(\mbf{G}_{\text{pie},1},\!\mbf{G}_{\text{pie},2})$, where 
	{\small
		\begin{align*}
			\mcl{T}&\!=\!\left[\!
			\begin{array}{ll}
				\mcl{T}_{1}+\mcl{T}_{1r}\mcl{D}_{rq}\mcl{C}_{q1} &\ \mcl{T}_{1r}\mcl{C}_{r2} \\
				\mcl{T}_{2q}\mcl{C}_{q1} &\ \mcl{T}_{2}
			\end{array}\!
			\right],	\hspace*{0.4cm}
			\mcl{T}_{w}\!=\!\left[\!
			\begin{array}{l}
				\mcl{T}_{1w} + \mcl{T}_{1r}\mcl{D}_{rq}\mcl{D}_{qw} \\
				\mcl{T}_{2q}\mcl{D}_{qw}
			\end{array}\!
			\right],	\\
			\mcl{A}&\!=\!\left[\!
			\begin{array}{ll}
				\mcl{A}_{1}+\mcl{B}_{1r}\mcl{D}_{rq}\mcl{C}_{q1} &\ \mcl{B}_{1r}\mcl{C}_{q2} \\
				\mcl{B}_{2q}\mcl{C}_{q1} &\ \mcl{A}_{2}
			\end{array}\!
			\right],	
			\hspace*{0.2cm}\mcl{B}_{w}\!=\!\left[\!
			\begin{array}{l}
				\mcl{B}_{1w} + \mcl{B}_{1r}\mcl{D}_{rq}\mcl{D}_{qw} \\
				\mcl{B}_{2q}\mcl{D}_{qw}
			\end{array}\!
			\right],	\\
			\mcl{C}_{z}&\!=\!\left[\!
			\begin{array}{ll}
				\mcl{C}_{z1}+\mcl{D}_{zr}\mcl{D}_{rq}\mcl{C}_{q1} &\ \mcl{D}_{zr}\mcl{C}_{r2}
			\end{array}\!
			\right],	\hspace*{0.2cm}
			\mcl{D}_{zw}=\mcl{D}_{zw}+\mcl{D}_{zr}\mcl{D}_{rq}\mcl{D}_{qw}.
		\end{align*}
	}
	Then, $(\smallbmat{\mbf{v}\\\mbs{\psi}},\mbf{z})$ solves the PIE defined by $\mbf{G}_{\text{pie}}$ with initial conditions $\smallbmat{\mbf{v}_{0}\\\mbs{\psi}_{0}}$ and input $\mbf{w}$ if and only if $(\mbf{v},\mbf{z},\mbf{q})$ and $(\mbs{\psi},\mbf{r})$ solve the PIEs defined by $\mbf{G}_{\text{pie},1}$ and $\mbf{G}_{\text{pie},2}$ with initial conditions $\mbf{v}_{0}$ and $\mbs{\psi}_{0}$ and inputs $(\mbf{w},\mbf{r})$ and $\mbf{q}$, respectively, where
	\begin{align}\label{eq:qr_signals}
		\mbf{q}(t)&=\mcl{C}_{q1}\mbf{v}(t) + \mcl{D}_{qw}\mbf{w}(t), \\		\mbf{r}(t)&=\mcl{C}_{r2}\mbs{\psi}(t)+\mcl{D}_{rq}\bl(\mcl{C}_{q1}\mbf{v}(t)+\mcl{D}_{qw}\mbf{w}(t)\br).	\notag
	\end{align}
	
\end{prop}

\begin{proof}
	A proof is given in Block~\ref{block:PIE_interconnection_proof}.
\end{proof}

\begin{block*}	
	Let inputs $(\mbf{w},\mbf{r})$ and $\mbf{q}$ be given, and such that $(\mbf{v},\mbf{z},\mbf{q})$ and $(\mbs{\phi},\mbf{r})$ solve the PIEs defined by $\mbf{G}_{\text{pie},1}$ and $\mbf{G}_{\text{pie},2}$ with initial conditions $\mbf{v}_{0}$ and $\mbs{\psi}_{0}$, respectively. Then, $\smallbmat{\mbf{v}(0)\\\mbs{\psi}(0)}=\smallbmat{\mbf{v}_{0}\\\mbs{\psi}_{0}}$, and therefore $\smallbmat{\mbf{v}(0)\\\mbs{\psi}(0)}$ satisfies the initial conditions defined by $\smallbmat{\mbf{v}_{0}\\\mbs{\psi}_{0}}$. In addition, the outputs $\mbf{q}$ to $\mbf{G}_{\text{pie}}$ and $\mbf{r}$ to $\mbf{G}_{\text{pie},2}$ will satisfy Eqn.~\eqref{eq:qr_signals}, by definition of the PIEs defined by $\mbf{G}_{\text{pie},1}$ and $\mbf{G}_{\text{pie},2}$. Finally, by definition of the operators $\mbf{G}_{\text{pie}}$, at any time $t\geq 0$, the signals $\bl(\smallbmat{\mbf{v}\\\mbs{\psi}},\mbf{z}\br)$ will satisfy
	\begin{flalign}\label{eq:PIE_interconnection_dynamics}
		0&=
		\srbmat{\mcl{T}\smallbmat{\mbf{v}_{t}(t)\\ \mbs{\psi}_{t}(t)}\\ \mbf{z}(t)} -
		\left[\!
		\begin{array}{ll}
			\mcl{A}&\mcl{B}_{w}\\ \mcl{C}_{z}&\mcl{D}_{zw}
		\end{array}\!
		\right]
		\srbmat{\smallbmat{\mbf{v}(t)\\ \mbs{\psi}(t)}\\\mbf{w}(t)} \\
		&=
		\left[\!
		\begin{array}{rcrcr}
			(\mcl{T}_{1}+\mcl{T}_{1r}\mcl{D}_{rq})\mbf{v}_{t}(t) &\!\!\!\!+\!\!\!\!&  \mcl{T}_{1r}\mcl{C}_{r2}\mbs{\psi}_{t}(t) &\!\!\!\!+\!\!\!\!& (\mcl{T}_{1w}+\mcl{T}_{1r}\mcl{D}_{rq}\mcl{D}_{qw})\mbf{w}_{t}(t) \\
			\mcl{T}_{2q}\mcl{C}_{q1}\mbf{v}_{t}(t) &\!\!\!\!+\!\!\!\!&  \mcl{T}_{2}\mbs{\psi}_{t}(t) &\!\!\!\!+\!\!\!\!&  \mcl{T}_{2q}\mcl{D}_{qw}\mbf{w}_{t}(t) \\
			\mbf{z}(t)
		\end{array}\!
		\right] \nonumber\\
		&\hspace*{2.0cm}-	
		\left[\!
		\begin{array}{rcrcr}
			(\mcl{A}_{1}+\mcl{B}_{1r}\mcl{D}_{rq})\mbf{v}(t) &\!\!\!\!+\!\!\!\!&  \mcl{B}_{1r}\mcl{C}_{r2}\mbs{\psi}(t) &\!\!\!\!+\!\!\!\!& (\mcl{B}_{1w}+\mcl{B}_{1r}\mcl{D}_{rq}\mcl{D}_{qw})\mbf{w}(t) \\
			\mcl{B}_{2q}\mcl{C}_{q1}\mbf{v}(t) &\!\!\!\!+\!\!\!\!&  \mcl{A}_{2}\mbs{\psi}(t) &\!\!\!\!+\!\!\!\!&  \mcl{B}_{2q}\mcl{D}_{qw}\mbf{w}(t) \\
			(\mcl{C}_{z1}+\mcl{D}_{zr}\mcl{D}_{rq})\mbf{v}(t)
			(\mcl{C}_{z1}+\mcl{D}_{zr}\mcl{D}_{rq})\mbf{v}(t) &\!\!\!\!+\!\!\!\!&  \mcl{D}_{zr}\mcl{C}_{r2}\mbs{\psi}(t) &\!\!\!\!+\!\!\!\!& (\mcl{D}_{zw}+\mcl{B}_{zr}\mcl{D}_{rq}\mcl{D}_{qw})\mbf{w}(t) 
		\end{array}\!
		\right] \nonumber\\
		&=
		\left[\!
		\begin{array}{lll}
			\mcl{T}_{1}\mbf{v}_{t}(t) ~+\!\!\!& \mcl{T}_{1w}\mbf{w}_{t}(t) ~+\!\!\!\!\!& \mcl{T}_{1r}\mbf{r}_{t}(t) \\
			\mcl{T}_{2}\mbs{\phi}_{t}(t) ~+\!\!\!\!\!& \mcl{T}_{2q}\mbf{q}_{t}(t) \\
			\mbf{z}(t)
		\end{array}\!
		\right]
		-	
		\left[\!
		\begin{array}{lll}
			\mcl{A}_{1}\mbf{v}(t) ~+\!\!\!\!\!& \mcl{B}_{1w}\mbf{w}(t) ~+\!\!\!\!\!& \mcl{B}_{1r}\mbf{r}(t) \\
			\mcl{A}_{2}\mbs{\psi}(t) ~+\!\!\!\!\!&
			\mcl{B}_{2q}\mbf{q}(t) \\
			\mcl{C}_{z1}\mbf{v}(t) ~+\!\!\!\!\!& \mcl{D}_{zw}\mbf{w}(t) ~+\!\!\!\!\!& \mcl{D}_{zr}\mbf{r}(t)
		\end{array}\!	 
		\right].
	\end{flalign}
	Hence, at any time $t\geq 0$, $\bl(\smallbmat{\mbf{v}(t)\\\mbs{\psi}(t)},\mbf{z}(t),\mbf{w}(t)\br)$ satisfies the PIE defined by $\mbf{G}_{\text{pie}}$. It follows that, $\bl(\smallbmat{\mbf{v}\\\mbs{\psi}},\mbf{z}\br)$ solves the PIE defined by $\mbf{G}_{\text{pie}}$, with input $\mbf{w}$, and initial conditions $\smallbmat{\mbf{v}_{0}\\\mbs{\psi}_{0}}$. 
	
	Conversely, let now an input $\mbf{w}$ be given, and suppose that $\bl(\smallbmat{\mbf{v}\\\mbs{\psi}},\mbf{z}\br)$ solves the PIE defined by $\mbf{G}_{\text{pie}}$, with initial conditions $\smallbmat{\mbf{v}_{0}\\\mbs{\psi}_{0}}$. Then, $\mbf{v}(0)=\mbf{v}_{0}$ and $\mbs{\psi}(0)=\mbs{\psi}_{0}$, and therefore $\mbf{v}$ and $\mbs{\psi}$ satisfy the initial conditions defined by $\mbf{v}_{0}$ and $\mbs{\psi}_{0}$. Moreover, defining $\mbf{q}$ and $\mbf{r}$ as in~\eqref{eq:qr_signals}, at any time $t\geq 0$, the identities in Eqn.~\eqref{eq:PIE_interconnection_dynamics} will be satisfied. By these identities, it follows that $\bl(\mbf{v}(t),(\mbf{z}(t),\mbf{q}(t)),(\mbf{w}(t),\mbf{r}(t))\br)$ and $\bl(\mbs{\psi}(t),\mbf{r}(t),\mbf{q}(t)\br)$ satisfy the PIEs defined by $\mbf{G}_{\text{pie},1}$ and $\mbf{G}_{\text{pie},2}$.
	Thus, $(\mbf{v},\mbf{z},\mbf{q})$ and $(\mbs{\phi},\mbf{r})$ solve the PIEs defined by $\mbf{G}_{\text{pie},1}$ and $\mbf{G}_{\text{pie},2}$ with initial conditions $\mbf{v}_{0}$ and $\mbs{\psi}_{0}$, respectively.

	\caption{Proof of Proposition~\ref{prop:PIE_interconnection_full}.}\label{block:PIE_interconnection_proof}
\end{block*}

\section{A PIE Representation of 1D ODE-PDE Systems with Delay}\label{appx:PIEmap}
%

In this section, we provide the main technical contribution of this paper, showing that for suitably well-posed linear 1D ODE-PDE systems, with delay in either the ODE or in the PDE, there exists an equivalent PIE representation. To reduce notational complexity, we will not explicitly derive the parameters defining this PIE representation in full detail here, instead leveraging results from earlier papers. In particular, in Subsection~\ref{appx:PIEmap:ODEdelay}, we repeat the result from~\cite{peet2021DDEs_PIEs}, showing that an equivalent 1D PIE representation exists for any linear ODE with constant delays. In Subsection~\ref{appx:PIEmap:PDEdelay}, we then use the results from~\cite{shivakumar2022GPDE_Arxiv} to prove that, for any well-posed, linear, 1D PDE with constant delays, there exists an equivalent 2D PIE representation. Finally, in Subsection~\ref{appx:PIEmap:ODEPDEdelay}, we combine these results, proving that any well-posed, linear, ODE-PDE system with constant delay can be equivalently represented as a PIE as well.



\subsection{A PIE Representation of ODEs with Delay}\label{appx:PIEmap:ODEdelay}

In~\cite{peet2021DDEs_PIEs}, it was shown that a general class of linear Delay Differential Equations (DDEs) can be equivalently represented as PIEs. In this paper, we consider only a restricted subclass of such DDEs, taking the form
\begin{align}\label{eq:delayODE}
	\srbmat{\dot{u}(t)\\z(t)}&=
	\left[\!{\small
		\begin{array}{ll}
			A&B_{w}\\
			C_{z}&0
	\end{array}}\!
	\right]
	\srbmat{u(t)\\w(t)} +
	\sum_{j=1}^{K}
	\left[\!{\small
		\begin{array}{ll}
			A_{j}	\\
			C_{z,j}
	\end{array}}\!
	\right]u(t-\tau_{j}),
\end{align}
where $0<\tau_{1}<\hdots<\tau_{K}$ are the delays, and where $u(t)\in\R^{n_u}$, $z(t)\in\R^{n_z}$ and $w(t)\in\R^{n_{w}}$ are all finite-dimensional.

\paragraph*{\textbf{Example}}
Consider the ODE with delay
\begin{align}\label{eq:example_DDE}
	\dot{u}(t)&=-u(t) + u(t-\tau), \\
	z(t)&=u(t).	\notag
\end{align}
Defining $A=-1$, $C_{z}=1$, $A_{1}=1$ and $C_{z,1}=0$, this system can be represented as in~\eqref{eq:delayODE}, letting $\tau_{1}=\tau$ and $n_w=0$.


\subsubsection*{Expanding the Delays}
To derive a PIE representation associated to the DDE~\eqref{eq:delayODE}, we first introduce delayed states $\mbs{\phi}_{j}(t,s)=u(t-\tau_{j}s)$, defining an equivalent ODE-PDE representation of the system as
\begin{align}\label{eq:delayODE_transport}
	\srbmat{\dot{u}(t)\\z(t)}&\!=\!
	\left[\!{\small
		\begin{array}{lll}
			A&A_{\text{d}}&B_{w}\\
			C_{z}&C_{z\text{d}}&0
	\end{array}}\!
	\right]
	\srbmat{u(t)\\\mbs{\phi}(t,1)\\w(t)} \\
	\partial_{t}\mbs{\phi}_{j}(t,s)&\!=\! -\frac{1}{\tau_{j}}\partial_{s}\mbs{\phi}_{j}(t,s),
	\quad
	\mbs{\phi}_{j}(t,0)\!=\!u(t),	\hspace*{0.1cm}
	{\small
		\begin{array}{r}
			s\in\Omega_{0}^{1}, \\[0.2em]
			j\in\{1,\hdots,K\},
	\end{array}}
	\notag
\end{align}
where $\srbmat{A_{\text{d}}\\ C_{z\text{d}}}\!:=\!
\left[\!{\small
	\begin{array}{lll}
		A_{1}\!&\!\hdots\!&\!A_{K}\\C_{z,1}\!&\!\hdots \!&\!C_{z,K}
\end{array}}\!
\right]$ and $\mbs{\phi}(t,s)\!:=\!\srbmat{\mbs{\phi}_{1}(t,s)\\:\\\mbs{\phi}_{K}(t,s)}$.
For $u\in\R^{n_u}$, we denote the domain of the delayed states $\mbs{\phi}_{j}$ as
\begin{align*}
	Y_{u}:=\left\{
	\mbs{\phi}\in H_1^{n_u}[\Omega_{0}^{1}]\ \bbr\rvert\ \Delta_{s}^{0}\mbs{\phi}=u
	\right\}
\end{align*}
where for $\mbf{v}\in H_{\text{k}}[\Omega_{ac}^{bd}]$, we denote the Dirac delta operators
\[
[\Delta_{x}^{a} \mbf{v}](y):=\mbf{v}(a,y)\quad \text{and}\quad [\Delta_{y}^{c} \mbf{v}](x):=\mbf{v}(x,c).
\]
We define $\bar{Y}_{u}^{K}=\overbrace{Y_{u}\times\hdots\times Y_{u}}^{K\text{ times}}$ as the domain of the full state $\mbs{\phi}$. Finally, collecting the delays as $\mbs{\tau}=(\tau_{1},\hdots,\tau_{K})$, and the parameters defining this system as
\begin{align*}
	\mbf{G}_{\text{dde}}&=
	\left[\!{\small
		\begin{array}{lll}
			A&A_{\text{d}}&B_{w}\\
			C_{z}&C_{z\text{d}}&0
	\end{array}}\!
	\right]
	\!\in\!
	\left[\!{\small
		\begin{array}{lll}
			\R^{n_u\times n_u}& \R^{n_u\times Kn_u}& \R^{n_u\times n_w}\\
			\R^{n_z\times n_u}& \R^{n_u\times Kn_u}& \R^{n_z\times n_w}
	\end{array}}\!
	\right],
\end{align*}
we define solutions to the DDE as follows.

\begin{defn}[Solution to the DDE]
	For a given input signal $w$ and initial conditions $(u_{0},\mbs{\phi}_{0})\in\R^{n_u}\times \bar{Y}_{u_{0}}^{K}$, we say that $(u,\mbs{\phi},z)$ is a solution to the DDE defined by $\{\mbf{G}_{\text{dde}},\mbs{\tau}\}$ if $(u,\mbs{\phi})$ is Frech\'et differentiable, $(u(0),\mbs{\phi}(0))=(u_{0},\mbs{\phi}_{0})$, and for all $t\geq0$, $\bl(u(t),\mbs{\phi}(t),z(t),w(t)\br)$ satisfies Eqn.~\eqref{eq:delayODE_transport}.
\end{defn}


\subsubsection*{Deriving a PIE Representation}

To derive a PIE representation associated to the ODE-PDE~\eqref{eq:delayODE_transport}, we first define a \textit{fundamental state} $\mbs{\psi}\in L_2^{Kn_u}[\Omega_{0}^{1}]$ associated to the delayed states $\mbs{\phi}$. In particular, this fundamental state must be free of the BCs $\mbs{\phi}_{j}(t,0)=u(t)$ imposed upon the delayed state $\mbs{\phi}\in \bar{Y}_{u}^{K}$. To achieve this, we take the derivative of the state $\mbs{\phi}$ along $\Omega_{0}^{1}$ as
\begin{align*}
	\mbs{\psi}(t,s)=\partial_{s}\mbs{\phi}(t,s).
\end{align*}
Then, imposing the BCs, $\mbs{\phi}$ can be expressed in terms of the fundamental state $\mbs{\psi}$ and $u$ using PI operators as
\begin{align*}
	\mbs{\phi}_{j}(s)&\!=\!\mbs{\phi}_{j}(0)\!+\!\int_{0}^{s}\!\bl(\partial_{s}\mbs{\phi}_{j}\br)(\theta)d\theta
	=u\!+\!\!\int_{0}^{s}\!\mbs{\psi}_{j}(\theta)d\theta, &	s&\in\Omega_{0}^{1}.
\end{align*}
Using this relation, an equivalent PIE representation of the ODE-PDE~\eqref{eq:delayODE_transport} can be easily defined, as shown in e.g.~\cite{peet2021DDEs_PIEs}.

\begin{cor}\label{cor:PIEmap_DDE}[PIE Representation of DDE]
	Let the linear map $\mcl{L}_{\text{pie}\leftarrow\text{dde}}$ be as defined in Lemma~4 of~\cite{peet2021DDEs_PIEs}, and let $\{\mcl{T},0,\mcl{A},\mcl{B}_{w},\mcl{C}_{z},0\}=\mbf{G}_{\text{pie}}=\mcl{L}_{\text{pie}\leftarrow\text{dde}}(\mbf{G}_{\text{dde}},\mbs{\tau})$.
	Then, for any input $w$, $(\smallbmat{u\\\mbs{\psi}},z)$ is a solution to the 1D PIE defined by $\mbf{G}_{\text{pie}}$ with initial conditions $(u_{0},\mbs{\psi}_{0})\in \R^{n_u}\times L_2^{Kn_u}[\Omega_{0}^{1}]$ if and only if $(\mcl{T}\smallbmat{u\\\mbs{\psi}},z)$ is a solution to the DDE defined by $\mbf{G}_{\text{dde}}$ with initial conditions $\smallbmat{u_{0}\\\mbs{\phi}_{0}}=\mcl{T}\smallbmat{u_{0}\\\mbs{\psi}_{0}}$.
	
\end{cor}

\begin{proof}
	A proof is given in~\cite{peet2021DDEs_PIEs}, Lemma~4.
\end{proof}

\paragraph*{\textbf{Example}}
For the delayed ODE System~\eqref{eq:example_DDE}, we define the delayed state $\mbs{\phi}(t,s):=u(t-s\tau)$, and fundamental state $\mbs{\psi}(t,s):=\partial_{s}\mbs{\phi}(t,s)$. Then, the ODE with delay can be equivalently represented as a PIE as
{
	\begin{align*}
		\dot{u}(t)&=\int_{0}^{1}\mbs{\psi}(t,s)ds, &
		\dot{u}(t)+\int_{0}^{s}\mbs{\psi}_{t}(t,\theta)d\theta&=-\frac{1}{\tau}\mbs{\psi}(t,s), \\
		z(t)&=u(t).
\end{align*}}%

\subsection{A PIE Representation of 1D PDEs with Delay}\label{appx:PIEmap:PDEdelay}

Having shown that an equivalent 1D PIE representation exists for suitable ODEs with delay, in this subsection, we show that an equivalent 2D PIE representation exists for any well-posed, linear 1D PDE with delay. In particular, we consider a system of the form
\begin{align}\label{eq:delayPDE}
	\srbmat{\mbf{u}_{t}(t)\\w(t)}&=
	\left[\!{\small
		\begin{array}{lll}
			A_{\text{p}}& A_{\text{b}} & B_{z}\\
			\smallint_{a}^{b}[C_{w\text{p}}] & C_{w\text{b}} & B_{wz}
	\end{array}}\!	
	\right]
	\srbmat{\bl(\mscr{D}_{\text{int}}\mbf{u}\br)(t)\\\bl(\Lambda_{\text{bf}}\mbf{u}\br)(t)\\ z(t)} \\
	&\hspace*{0.1cm}+\sum_{j=1}^{K}	
	\left[\!{\small
		\begin{array}{lll}
			A_{\text{p},j}& A_{\text{b},j} & B_{z,j}\\
			\smallint_{a}^{b}[C_{w\text{p},j}] & C_{w\text{b},j} & D_{wz,j}
	\end{array}}\!	
	\right]\!
	\srbmat{\bl(\mscr{D}_{\text{int}}\mbf{u}\br)(t-\tau_{j})\\\bl(\Lambda_{\text{bf}}\mbf{u}\br)(t-\tau_{j})\\z(t-\tau_{j})},	\notag\\
	\text{with BCs}\hspace*{-0.1cm}&\qquad 0=\left[\!{\small
		\begin{array}{lll}
			\smallint_{a}^{b}[E_{\text{p}}] & E_{\text{b}} & E_{z}
		\end{array}\!}
	\right]
	\srbmat{\mscr{D}_{\text{int}}\mbf{u}(t)\\\Lambda_{\text{bf}}\mbf{u}(t)\\ z(t)}, \notag
\end{align}
where the PDE state variables are distinguished based on their order of differentiability as
\begin{align*}
	\mbf{u}(t)=\srbmat{\mbf{u}_1(t)\\\mbf{u}_{2}(t)\\\mbf{u}_{3}(t)}\in\srbmat{L_2^{n_1}[\Omega_{a}^{b}]\\H_1^{n_2}[\Omega_{a}^{b}]\\H_2^{n_3}[\Omega_{a}^{b}]}=:U^{\text{n}_{\text{p}}}[\Omega_{a}^{b}],
\end{align*}
for $\text{n}_{\text{p}}=(n_1,n_2,n_3)\in\N^3$, defining an associated differential operator $\mscr{D}_{\text{int}}:U^{\text{n}_{\text{p}}}\rightarrow L_2^{n_{\text{int}}}$ and boundary Dirac operator $\Lambda_{\text{bf}}:U^{\enn{\text{p}}}\rightarrow L_2^{n_{\text{bf}}}$ for $n_{\text{int}}=n+1+2+n_2+3n_3$ and $n_{\text{bf}}=2(n_2+2n_3)$ as
\begin{align*}
	\mscr{D}_{\text{int}}&:=\bmat{I&0&0\\0&I&0\\0&0&I\\0&\partial_{x}&0\\0&0&\partial_{x}\\0&0&\partial_{x}^2},&
	&\begin{array}{l}		
		\Lambda_{\text{bf}}:=\bmat{\Delta_{x}^{a}\mscr{D}_{\text{bc}}\\\Delta_{x}^{b}\mscr{D}_{\text{bc}}}, \\
		\text{where}\\
		\mscr{D}_{\text{bc}}:=\bmat{0&I&0\\0&0&I\\0&0&\partial_{x}}.
	\end{array}	
\end{align*}

\paragraph*{\textbf{Example}}
Consider a delayed reaction-diffusion PDE
\begin{align}\label{eq:example_DPDE}
	\mbf{u}_{t}(t,x)&=\mbf{u}_{xx}(t,x)+10\mbf{u}(t,x)-3\mbf{u}(t-\tau,x), & x&\in\Omega_{0}^{1}, \notag\\
	\mbf{u}(t,0)&=0,\qquad \mbf{u}(t,1)=z(t).
\end{align}
In this PDE, $\mbf{u}(t)\in U^{(0,0,1)}[\Omega_{0}^{1}]=H_{2}[\Omega_{0}^{1}]$, and therefore $\mscr{D}_{\text{int}}=\srbmat{I\\\partial_{x}\\\partial_{x}^{2}}$ and $\Lambda_{\text{bf}}=\smallbmat{\Delta_{x}^{0}\\\Delta_{x}^{0}\partial_{x}\\\Delta_{x}^{1}\\\Delta_{x}^{1}\partial_{x}}$. Defining $A_{\text{p}}:=\bmat{1&0&10}$, $A_{\text{p},1}:=\bmat{-3&0&0}$, $E_{\text{b}}:=\smallbmat{1&0&0&0\\0&0&1&0}$, $E_{z}:=\smallbmat{0\\-1}$, and setting all other parameters equal to zero, the delayed PDE~\eqref{eq:example_DPDE} can be represented as in~\eqref{eq:delayPDE}, with $\tau_1=\tau$ and $n_{w}=0$.


\subsubsection*{Expanding the Delays}

To derive a PIE representation of the delayed PDE~\eqref{eq:delayPDE}, we first represent the system as the interconnection of a 1D PDE
\begin{align}\label{eq:delayPDE_PDE}
	&\srbmat{\mbf{u}_{t}(t)\\w(t)\\\mbf{q}(t)}\!=\!
	\left[\!{\small
		\begin{array}{llll}
			A_{\text{p}}  & A_{\text{b}}  & B_{z} & \mcl{B}_{r}\\
			\mcl{C}_{w\text{p}} & C_{w\text{b}} & D_{wz} & \mcl{D}_{wr} \\
			C_{q\text{p}} & C_{q\text{b}} & D_{qz} & 0
	\end{array}}\!	
	\right]
	\srbmat{\bl(\mscr{D}_{\text{int}}\mbf{u}\br)(t)\\\bl(\Lambda_{\text{bf}}\mbf{u}\br)(t)\\z(t)\\\mbf{r}(t)}, \\[-0.4em]
	&\ 0=\left[\!{\small
		\begin{array}{lll}
			\smallint_{0}^{1}[E_{\text{p}}] & E_{\text{b}} & E_{z}
		\end{array}\!}
	\right]
	\srbmat{\mscr{D}_{\text{int}}\mbf{u}(t)\\\Lambda_{\text{bf}}\mbf{u}(t)\\z(t)},	\notag
\end{align}
with a 2D PDE,
\begin{align}\label{eq:delayPDE_transport}	
	&\partial_{t}\mbs{\phi}_{j}(t,s)=-(1/\tau_{j})\ \partial_{s}\mbs{\phi}_{j}(t,s), \hspace*{2.0cm} 	s\in\Omega_{0}^{1}, \\
	&\mbf{r}_{j}(t)=\srbmat{\bl(\mscr{D}_{\text{int}}\mbs{\phi}_{u,j}\br)(t,1)\\\bl(\Lambda_{\text{bf}}\mbs{\phi}_{u,j}\br)(t,1)\\\mbs{\phi}_{z,j}(t,1)}, \hspace*{2.0cm} j\in\{1,\hdots,K\}, \notag\\[-0.4em]
	&\mbs{\phi}_{j}(t,0)=T_{q}\mbf{q}(t),\quad
	0=\left[\!{\small
		\begin{array}{lll}
			\smallint_{a}^{b}[E_{\text{p}}] & \!E_{\text{b}} & \!E_{z}
		\end{array}\!}
	\right]
	\srbmat{\bl(\mscr{D}_{\text{int}}\mbs{\phi}_{u,j}\br)(t,s)\\\bl(\Lambda_{\text{bf}}\mbs{\phi}_{u,j}\br)(t,s)\\\mbs{\phi}_{z,j}(t,s)}, \notag
\end{align}
where $\mbf{r}=\srbmat{\mbf{r}_1\\:\\\mbf{r}_{K}}$,  $\mbs{\phi}_{j}(t,s)=\srbmat{\mbs{\phi}_{u,j}(t,s)\\\mbs{\phi}_{z,j}(t,s)}=\srbmat{\mbf{u}(t+\tau_{j}s)\\z(t+\tau_{j}s)}$, and where we define the new parameters such that
\begin{align*}
	&\srbmat{\mcl{B}_{r}\\\mcl{D}_{wr}}=\left[\
	\left[\!{\small
		\begin{array}{lll}
			A_{\text{p},j}& A_{\text{b},j} & B_{z,j}\\
			\smallint_{a}^{b}[C_{w\text{p},j}] & C_{w\text{b},j} & D_{wz,j}
	\end{array}}\!	
	\right]_{j=(1,\hdots,K)}\ 
	\right],	\\
	&\mcl{C}_{wp}=\smallint_{a}^{b}[C_{w\text{p}}],\qquad \bmat{C_{r\text{p}}&C_{r\text{b}}&D_{rz}}=I_{(n_{\text{int}}+n_{\text{bf}}+n_{z})},	\\
	&T_{q}\mbf{q}(t)=\bmat{T_{vq}&0&T_{vz}}\srbmat{\bl(\mscr{D}_{\text{int}}\mbf{u}\br)(t)\\\bl(\Lambda_{\text{bf}}\mbf{u})(t)\\z(t)}=\srbmat{\mbf{u}(t)\\z(t)}.
\end{align*}
We denote the parameters defining the PDE~\eqref{eq:delayPDE_PDE} as
\begin{align}\label{eq:DPDE_params}
	\mbf{G}_{\text{pde}}&:=
	\left[\!{\small
		\begin{array}{lll}
			A_{\text{p}}  & A_{\text{b}}  & B_{z} \\
			C_{w\text{p}} & C_{w\text{b}} & D_{wz} 
	\end{array}}\!	
	\right] &
	\mbf{G}_{\text{pde,d}}&:=
	\left[\!{\small
		\begin{array}{l}
			\mcl{B}_{r}\\
			\mcl{D}_{wr}
	\end{array}}\!	
	\right], \\
	\mbf{G}_{\text{bc}}&:=
	\{E_{\text{p}}, E_{\text{b}}, E_{z}\}.
	& \mbf{G}_{\text{pdde}}&:=\{\mbf{G}_{\text{pde}},\mbf{G}_{\text{pde,d}}\}. \notag
\end{align}
Given parameters $\mbf{G}_{\text{bc}}$, we define the domain of $\mbf{u}$ as
\begin{align}\label{eq:Xmap}
	X_{z}^{\enn{\text{p}}} := \left\{
	\mbf{u}\in U^{\text{n}_{\text{p}}} ~\bbbr\rvert~
	\left[\!{\small
		\begin{array}{lll}
			\smallint_{a}^{b}[E_{\text{p}}] & E_{\text{b}} & E_{z}
		\end{array}\!}
	\right]\!
	\srbmat{\mscr{D}_{\text{int}}\mbf{u}\\\Lambda_{\text{bf}}\mbf{u}\\ z}\!=\!0
	\right\},
\end{align}
so that $\mbf{u}\in X_{z}$ only if $\mbf{u}$ satisfies the BCs defined by $\mbf{G}_{\text{bc}}$.
Similarly, for a given input $\mbf{q}$, we define the domain of the delayed states $\mbs{\phi}_{j}$ at any time as
\begin{align}\label{eq:Yset_PDE}
	\mbf{Y}_{\mbf{q}}\!:=\!\bbbbl\{&
	\srbmat{\mbs{\phi}_{u}\\\mbs{\phi}_{z}}\in \srbmat{V^{\enn{p}}[\Omega_{0a}^{1b}]\\H_{1}^{n_z}[\Omega_{0}^{1}]}\ \bbbbr\rvert\ \Delta_{s}^{0}\mbs{\phi}=T_{q}\mbf{q},\\[-1.0em]
	&\hspace*{2.75cm}
	0=\left[\!{\small
		\begin{array}{lll}
			\smallint_{a}^{b}[E_{\text{p}}] &\! E_{\text{b}} & E_{z}
		\end{array}\!}
	\right]
	\srbmat{\mscr{D}_{\text{int}}\mbs{\phi}_{u}\\\Lambda_{\text{bf}}\mbs{\phi}_{u}\\\mbs{\phi}_{z}}	
	\bbbbr\}, \nonumber
\end{align}
where the components of $\mbs{\phi}_{u,j}(t,s)=\mbf{u}(t-\tau_{j}s)$ are distinguished based on their order of differentiability as
\begin{align*}
	\mbs{\phi}_{u,j}=\srbmat{\mbs{\phi}_{1,j}\\\mbs{\phi}_{2,j}\\\mbs{\phi}_{3,j}}\in\srbmat{H_{(0,1)}[\Omega_{0a}^{1b}]\\H_{(1,1)}[\Omega_{0a}^{1b}]\\H_{(2,1)}[\Omega_{a0}^{b1}]}=:V^{\enn{p}}[\Omega_{0a}^{1b}].
\end{align*}
For the full state $\mbs{\phi}:=(\mbs{\phi}_1,\hdots,\mbs{\phi}_{K})$, we define the domain $\bar{\mbf{Y}}_{\mbf{q}}^{K}=\mbf{Y}_{\mbf{q}}\times\hdots\times \mbf{Y}_{\mbf{q}}$.

\begin{defn}[Solution to the DPDE]
	For a given input $z$ and given initial conditions $(\mbf{u}_{0},\mbs{\phi}_{0})\in X^{\enn{p}}\times \bar{\mbf{Y}}_{\mbf{q}_{0}}^{K}$ with $\mbf{q}_{0}=\smallbmat{(\mscr{D}_{\text{int}}\mbf{u}_{0})(t)\\(\Lambda_{\text{bf}}\mbf{u}_{0})(t)\\z(0)}$, we say that $((\mbf{u},\mbs{\phi}),w)$ is a solution to the delayed PDE defined by $\{\mbf{G}_{\text{pdde}},\mbf{G}_{\text{bc}},\mbs{\tau}\}$ if $(\mbf{u},\mbs{\phi})$ is Frech\'et differentiable, $(\mbf{u}(0),\mbs{\phi}(0))=(\mbf{u}_{0},\mbs{\phi}_{0})$, and there exist $(\mbf{q},\mbf{r})$ such that $\bl(\mbf{u}(t),\mbs{\phi}(t),w(t),z(t),\mbf{q}(t),\mbf{r}(t)\br)$ satisfies Eqns.~\eqref{eq:delayPDE_PDE} and~\eqref{eq:delayPDE_transport} for all $t\geq0$.
\end{defn}

\subsubsection*{Deriving a PIE Representation}

To show that the system defined by the 1D PDE~\eqref{eq:delayPDE_PDE} coupled to the 2D PDE~\eqref{eq:delayPDE_transport} can be equivalently represented as a PIE, we first show that each individual system can be represented as a PIE. Consider first the PDE~\eqref{eq:delayPDE_PDE}. To define the fundamental state associated to this PDE, we use a differential operator $\mscr{D}:X^{\enn{\text{p}}}\rightarrow L_2^{\|\enn{\text{p}}\|}$, where $\|\enn{\text{p}}\|=n_1+n_2+n_3$ for $\enn{\text{p}}=(n_1,n_2,n_3)\in\N^3$. In particular, we define
\begin{align}\label{eq:Dmap_uPDE}
	\mbf{v}=\srbmat{\mbf{v}_1\\\mbf{v}_2\\\mbf{v}_3}=
	\underbrace{\left[\!{\small
			\begin{array}{lll}
				I_{n_1}\\ &\partial_{x}\\&&\partial_{x}^2
		\end{array}}\!
		\right]}_{\mscr{D}}
	\srbmat{\mbf{u}_1\\\mbf{u}_2\\\mbf{u}_3}
	=\mscr{D}\mbf{u}.
\end{align}
Then, if the BCs defined by $\mbf{G}_{\text{bc}}$ are well-posed, by Theorem~10 in~\cite{shivakumar2022GPDE_Arxiv}, there exist PI operators $\mcl{T}$ and $\mcl{T}_{z}$ such that $\mbf{u}(t)=\mcl{T}\mbf{v}(t)+\mcl{T}_{z}\mbf{z}(t)\in X_{z}^{\enn{p}}$ for any $\mbf{v}\in L_2^{\|\enn{p}\|}$.
Using this relation, we can define a PIE representation of the PDE~\eqref{eq:delayPDE_PDE} for arbitrary inputs $\mbf{r}$.

\begin{lem}[PIE Representation of 1D PDE]\label{lem:PIEmap:PDE}
	Let parameters $\{\mbf{G}_{\text{pdde}},\mbf{G}_{\text{bc}}\}$ be as in~\eqref{eq:DPDE_params}, and such that $\mbf{G}_{\text{bc}}$ defines a well-posed set of BCs. Let $\mbf{G}_{\text{pie},0}\!=\!\{\mcl{T},\mcl{T}_{z},\mcl{A},\mcl{B}_{z},\mcl{C}_{w},\mcl{D}_{wz}\!\}$ define the PIE associated to the PDE~$\{\mbf{G}_{\text{pde}},\mbf{G}_{\text{bc}}\}$ without delay ($n_q=n_r=0$), as defined in Thm.~12 of~\cite{shivakumar2022GPDE_Arxiv}.
	Define
	\begin{align*}
		\mbf{G}_{\text{pie}}&\!=\!
		\left[\!{\small	\arraycolsep=2.5pt
			\begin{array}{lll}
				\mcl{T}&\mcl{T}_{z}&0\\
				\mcl{A}&\mcl{B}_{z}&\mcl{B}_{r}\\
				\mcl{C}_{w}&\mcl{D}_{wz}&\mcl{D}_{wr}\\
				\mcl{C}_{q}&\mcl{D}_{qz}&0
		\end{array}}\!
		\right],	&
		&{\small\begin{array}{l}
				\mcl{C}_{q}=C_{q\text{p}}\mcl{T}_{\text{int}}+C_{q\text{b}}\mcl{T}_{\text{bf}},\\[0.5em]
				\mcl{D}_{qz}=D_{qz}+ C_{q\text{p}}\mcl{T}_{\text{int},z}\\[0.25em]
				\hspace*{1.75cm}+\ C_{q\text{b}}\mcl{T}_{\text{bf},w},
		\end{array}}
	\end{align*}
	where $\srbmat{\mcl{B}_{r}\\\mcl{D}_{wr}}=\mbf{G}_{\text{pde,d}}$,~~ $\bmat{C_{r\text{p}}&C_{r\text{b}}&D_{rz}}=I_{n_{q}}$, and
	\begin{align}\label{eq:Tops_bf}
		\mcl{T}_{\text{int}}&=\mscr{D}_{\text{int}}\circ\mcl{T},	&
		\mcl{T}_{\text{bf}}&=\Lambda_{\text{bf}}\circ\mcl{T}, \notag\\
		\mcl{T}_{\text{int},z}&=\mscr{D}_{\text{int}}\circ\mcl{T}_{z}, &
		\mcl{T}_{\text{bf},z}&=\Lambda_{\text{bf}}\circ \mcl{T}_{z}, 
	\end{align}
	Then, $(\mbf{v},w,\mbf{q})$ is a solution to the PIE defined by $\mbf{G}_{\text{pie}}$ with inputs $(z,\mbf{r})$ and initial conditions $\mbf{v}_{0}\in L_2^{\|\enn{p}\|}[\Omega_{a}^{b}]$ if and only if $(\mbf{u},w,\mbf{q})$ with $\mbf{u}=\mcl{T}\mbf{v}+\mcl{T}_{z}z$ is a solution to the PDE~\eqref{eq:delayPDE_PDE} defined by $\{\mbf{G}_{\text{pdde}},\mbf{G}_{\text{bc}}\}$ with inputs $(z,\mbf{r})$ and initial conditions $\mbf{u}_{0}=\mcl{T}\mbf{v}_{0}+\mcl{T}_{z}z\in X_{z(0)}^{\enn{p}}$.
	
\end{lem}


\begin{proof}
	Defining $\mbf{G}_{\text{pie},0}$ as in Thm.~12 of~\cite{shivakumar2022GPDE_Arxiv}, $(\mbf{v},w)$ is a solution to the PIE defined by $\mbf{G}_{\text{pie},0}$ with initial conditions $\mbf{v}_{0}$ if and only if $(\mbf{u},z)$ with $\mbf{u}=\mcl{T}\mbf{v}+\mcl{T}_{z}z$ is a solution to the PDE defined by $\{\mbf{G}_{\text{pdde}},\mbf{G}_{\text{bc}}\}$ with $n_{q}=n_{r}=0$ and with initial conditions $\mbf{u}_{0}=\mcl{T}\mbf{v}_{0}$. Since the BCs in the PDE defined by $\{\mbf{G}_{\text{pdde}},\mbf{G}_{\text{bc}}\}$ do not depend on the input signal $\mbf{r}$, it follows that for any $\mbf{v}\in L_2^{\|\enn{p}\|}[\Omega_{a}^{b}]$, $\mbf{u}=\mcl{T}\mbf{v}+\mcl{T}_{z}z\in X_{z}^{\enn{p}}$ also satisfies the BCs of the PDE defined by $\{\mbf{G}_{\text{pdde}},\mbf{G}_{\text{bc}}\}$ with $n_{r},n_{q}\neq 0$. Moreover, $\mbf{v}$ satisfies the initial conditions defined by $\mbf{v}_{0}$ if and only if $\mbf{u}$ satisfies the initial conditions defined by $\mbf{u}_{0}$.

	Now, let $\mbf{v}\in L_2^{\|\enn{p}\|}$ be arbitrary, and let $\mbf{u}=\mcl{T}\mbf{v}+\mcl{T}_{z}z\in X_{z}^{\enn{p}}$. Since $\mbf{v}$ is a solution to the PIE defined by $\mbf{G}_{\text{pie},0}$ if and only if $\mbf{u}$ is a solution to the PDE defined by $\{\mbf{G}_{\text{pdde}},\mbf{G}_{\text{bc}}\}$, it follows that, for any $t\geq 0$,
	\begin{align*}
		&\srbmat{\mbf{u}_{t}(t)\\w(t)}
		-
		\left[\!\!{\small \arraycolsep=2.5pt
			\begin{array}{lll}		
				A_{\text{p}}&A_{\text{b}}&B_{z}\\
				C_{w\text{p}}&C_{w\text{b}}&D_{wz}
		\end{array}}\!\!	
		\right]\!
		\srbmat{\bl(\mscr{D}_{\text{int}}\mbf{u}\br)(t)\\\bl(\Lambda_{\text{bf}}\mbf{u}\br)(t)\\z(t)}
		-
		\srbmat{\mcl{B}_{r}\\\mcl{D}_{wr}}\mbf{r}(t) \\
		&\hspace*{1.25cm}=\!
		\srbmat{\mcl{T}\mbf{v}_{t}(t)+\mcl{T}_{z}\dot{z}(t)\\w(t)}
		-
		\left[\!{\small
			\begin{array}{lll}
				\mcl{A}&\mcl{B}_{z}&\mcl{B}_{r}\\
				\mcl{C}_{w}&\mcl{D}_{wz}&\mcl{D}_{wr}
		\end{array}}\!
		\right]
		\srbmat{\mbf{v}(t)\\z(t)\\\mbf{r}(t)} \notag
	\end{align*}
	Similarly, applying the definition of the operators $\mcl{C}_{q}$,
	\begin{align*}
		&\mbf{q}(t)-C_{q\text{p}}\bl(\mscr{D}_{\text{int}}\mbf{u}\br)(t) - C_{q\text{b}}\bl(\Lambda_{\text{bf}}\mbf{u}\br)(t) -D_{qz}z(t) \\
		&=\mbf{q}(t)-C_{q\text{p}}\bl(\mscr{D}_{\text{int}}[\mcl{T}\mbf{v}+\mcl{T}_{z}z]\br)(t) \\ &\qquad\qquad-C_{q\text{b}}\bl(\Lambda_{\text{bf}}[\mcl{T}\mbf{v}+\mcl{T}_{z}z]\br)(t) - D_{qz}z(t) \notag\\
		&\hspace*{4.25cm}=\mbf{q}(t)-\mcl{C}_{q}\mbf{v}(t) - \mcl{D}_{qz}z(t).	
	\end{align*}
	It follows that, $(\mbf{v},w,\mbf{q})$ satisfies the PIE defined by $\mbf{G}_{\text{pie}}$ if and only if $(\mcl{T}\mbf{v}+\mcl{T}_{z}z,\mbf{q})$ satisfies the PDE~\eqref{eq:delayPDE_PDE} defined by $\{\mbf{G}_{\text{pdde}},\mbf{G}_{\text{bc}}\}$.
	
\end{proof}

\paragraph*{\textbf{Example}}
For the PDE with delay defined by~\eqref{eq:example_DPDE}, the fundamental state is given by $\mbf{v}(t,x)=\partial_{x}^{2}\mbf{u}(t,x)$. Defining 
\begin{align}\label{eq:example_Tmap_DPDE}
	\bl(\!\mcl{T}\mbf{v}\br)(t,x)\!=\!\!\int_{a}^{x}\!\!\!\theta(x\!-\!1)\mbf{v}(t,\theta)d\theta \!+\! \!\int_{x}^{b}\!\!\! x(\theta\!-\!1)\mbf{v}(t,\theta)d\theta,
\end{align}
and $(\mcl{T}_{z}z)(t,x)=xz(t)$, we can retrieve the PDE state as $\mbf{u}(t)=(\mcl{T}\mbf{v})(t)+(\mcl{T}_{z}z)(t)$. Defining $\mbs{\phi}_{u}(t,s)=\mbf{u}(t-\tau s)$, the PDE~\eqref{eq:example_DPDE} can then be equivalently represented as
{
	\begin{align*}
		\mcl{T}_{z}\dot{z}(t) + \mcl{T}\mbf{v}_{t}(t) 
		&=\mbf{v}(t) + 10\mcl{T}_{z}z(t) +10\mcl{T}\mbf{v}(t) - 3\mbs{\phi}_{u}(t,1).
	\end{align*}
}%

Having shown that the PDE~\eqref{eq:delayPDE_PDE} can be equivalently represented as a PIE, we now show that the PDE~\eqref{eq:delayPDE_transport} can also be equivalently represented as a PIE. For this, we define the fundamental state associated to the states $\mbs{\phi}_{j}$ as
\begin{align}\label{eq:Dmap_vPDE}
	\mbs{\psi}_{j}=\srbmat{\mbs{\psi}_{u,j}\\\mbs{\psi}_{z,j}}=
	\underbrace{\left[\!{\small
			\begin{array}{ll}
				\partial_{s}\circ\mscr{D} \\
				& \partial_{s}
		\end{array}}\!
		\right]}_{\bar{\mscr{D}}_{v}}
	\srbmat{\mbs{\phi}_{u,j}\\\mbs{\phi}_{z,j}}
	=\bar{\mscr{D}}_{v}\mbs{\phi}_{j}, \\[-1.6em] \notag
\end{align}
where $\mscr{D}$ is as defined in~\eqref{eq:Dmap_uPDE}, so that $\mbf{v}=\mscr{D}\mbf{u}$ is the fundamental state associated to $\mbf{u}\in X_{z}$. Since the BCs imposed upon the states $\mbs{\phi}_{j}$ are defined by the same $\mbf{G}_{\text{bc}}$ as the BCs imposed upon $\mbf{u}$, the same PI operators $\mcl{T}$ and $\mcl{T}_{z}$ can also be used to derive the PIE representation of the PDE~\eqref{eq:delayPDE_transport}, as we prove in the following Lemma.

\begin{lem}\label{lem:PIEmap:PDE_transport}
	Let $\mbf{G}_{\text{bc}}$ define a well-posed set of boundary conditions, and let operators $\{\mcl{T},\mcl{T}_{z}\}$ be as defined Lemma~\ref{lem:PIEmap:PDE}. For $\mbs{\psi}_{j}=\smallbmat{\mbs{\psi}_{u,j}\\\mbs{\psi}_{z,j}}$, define the 2D PI operators
	{
		\begin{align*}
			\bl(\bar{\mcl{T}}_{v,j}\mbs{\psi}_{j}\br)(t,s)&=\!\int_{0}^{s}\!\srbmat{\mcl{T}\mbs{\psi}_{u,j}(t,\theta)+\mcl{T}_{z}\mbs{\psi}_{z,j}(t,\theta)\\\mbs{\psi}_{z,j}(t,\theta)}d\theta ,	\quad s\in\Omega_{0}^{1},\\
			\bl(\bar{\mcl{A}}_{v,j}\mbs{\psi}_{j}\br)(t,s)&=-\frac{1}{\tau_{j}}\srbmat{\mcl{T}\mbs{\psi}_{u,j}(t,s)+\mcl{T}_{z}\mbs{\psi}_{z,j}(t,s)\\\mbs{\psi}_{z,j}(t,s)}, \\
			\bl(\bar{\mcl{C}}_{r,j}\mbs{\psi}_{j}\br)(t)&=\!\int_{0}^{1}\!\srbmat{\mcl{T}_{\text{int}}\mbs{\psi}_{u,j}(t,s)+\mcl{T}_{\text{int},z}\mbs{\psi}_{z,j}(t,s)\\\mcl{T}_{\text{bf}}\mbs{\psi}_{u,j}(t,s)+\mcl{T}_{\text{bf},z}\mbs{\psi}_{z,j}(t,s)\\\mbs{\psi}_{z,j}(t,s)}ds,
		\end{align*}
	}%
	where $\mcl{T}_{\text{int}}$, $\mcl{T}_{\text{int},z}$, $\mcl{T}_{\text{bf}}$, and $\mcl{T}_{\text{bf},z}$ are as defined in~\eqref{eq:Tops_bf}. Then, for a given input $\mbf{q}=\smallbmat{\mscr{D}_{\text{int}}\mbf{q}_{u}(t)\\\Lambda_{\text{bf}}\mbf{q}_{u}(t)\\\mbf{q}_{z}(t)}$, $(\mbs{\psi}_{j},\mbf{r}_{j})$ is a solution to the PIE defined by $\mbf{G}_{\text{pie}}=\{\bar{\mcl{T}}_{v,j},T_{q},\bar{\mcl{A}}_{v,j},0,\bar{\mcl{C}}_{r,j},I_{n_q}\}$ with initial conditions $\mbs{\psi}_{0,j}\in L_{2}^{\|\enn{\text{p}}\|+n_{z}}$ if and only if $(\mbs{\phi}_{j},\mbf{r}_{j})$ with $\mbs{\phi}_{j}=\bar{\mcl{T}}_{v,j}\mbs{\psi}_{j}+T_{q}\mbf{q}$ is a solution to the $j$th PDE~\eqref{eq:delayPDE_transport} with initial conditions $\mbs{\phi}_{0,j}=\bar{\mcl{T}}_{v,j}\mbs{\psi}_{0,j}+T_{q}\mbf{q}(0)\in \mbf{Y}_{\mbf{q}(0)}$.
	
\end{lem}
\begin{proof}
	Since $\mbf{G}_{\text{bc}}$ is the same as in Lemma~\ref{lem:PIEmap:PDE}, letting $\{\mcl{T},\mcl{T}_{z}\}$ be as in Lemma~\ref{lem:PIEmap:PDE} as well, each $\mbs{\phi}_{j}(t,s)$ must satisfy
	\begin{align*}
		\mbs{\phi}_{u,j}(s)&=\mcl{T}\mscr{D}\mbs{\phi}_{u,j}(s) + \mcl{T}_{z}\mbs{\phi}_{z,j},	&	s&\in\Omega_{0}^{1}.
	\end{align*}
	By the fundamental theorem of calculus, it follows that
	{\begin{align*}
			\mbs{\phi}_{j}(s)&=\mbs{\phi}_{j}(0)+\int_{0}^{s}\srbmat{\bl(\partial_{s}\mcl{T}\mscr{D}\mbs{\phi}_{u,j}\br)(\theta) + \bl(\partial_{s}\mcl{T}_{z}\mbs{\phi}_{z,j}\br)(\theta)\\\bl(\partial_{s}\mbs{\phi}_{z,j}\br)(\theta)}d\theta \\
			&=T_{q}\mbf{q}+\int_{0}^{s}\srbmat{\bl(\mcl{T}\mbs{\psi}_{u,j}\br)(\theta) + \bl(\mcl{T}_{z}\mbs{\psi}_{z,j}\br)(\theta)\\\mbs{\psi}_{z,j}(\theta)} d\theta,
	\end{align*}}%
	where $\mbs{\psi}_{j}=\bar{\mscr{D}}_{v}\mbs{\phi}_{j}$ is as in~\eqref{eq:Dmap_vPDE}.	 Hence, for any $\mbs{\psi}_{j}(t)\in L_{2}^{\|\enn{\text{p}}\|+n_{z}}$, $\mbs{\phi}_{j}=\bar{\mcl{T}}_{v,j}\mbs{\psi}_{j}(t)+T_{q}\mbf{q}(t)\in \mbf{Y}_{\mbf{q}(t)}$ will satisfy the BCs defined by $\mbf{G}_{\text{bc}}$.
	
	Let now $(\mbs{\psi}_{j},\mbf{r}_{j})$ be a solution to the PIE defined by $\mbf{G}_{\text{pie}}=\{\bar{\mcl{T}}_{v,j},T_{q},\bar{\mcl{A}}_{v,j},0,\bar{\mcl{C}}_{r,j},I_{n_q}\}$ with initial conditions $\mbs{\psi}_{0,j}\in L_{2}^{\|\enn{\text{p}}\|+n_{z}}$ and input $\mbf{q}=\smallbmat{\mscr{D}_{\text{int}}\mbf{q}_{u}(t)\\\Lambda_{\text{bf}}\mbf{q}_{u}(t)\\\mbf{q}_{z}(t)}$. Then
	\begin{align*}
		\partial_{t}\mbs{\phi}_{j}(t,s)&=\bl(\partial_{t}\bar{\mcl{T}}_{v,j}\mbs{\psi}_{j}\br)(t,s)+T_{q}\mbf{q}_{t}(t) \\
		&=\bl(\bar{\mcl{A}}_{v,j}\mbs{\psi}_{j}\br)(t,s) \\
		&=-\frac{1}{\tau_{j}}\srbmat{\mcl{T}\mbs{\psi}_{u,j}(t,s)+\mcl{T}_{z}\mbs{\psi}_{z,j}(t,s)\\\mbs{\psi}_{z,j}(t,s)} \\
		&=-\frac{1}{\tau_{j}}\srbmat{\partial_{s}\mbs{\phi}_{u,j}(t,s)\\\partial_{s}\mbs{\phi}_{z,j}(t,s)} = -(1/\tau_{j})~\partial_{s}\mbs{\phi}_{j}(t,s),
	\end{align*}
	and
	\begin{align*}
		&\mbf{r}_{j}(t)=\bl(\bar{\mcl{C}}_{r,j}\mbs{\psi}_{j}\br)(t)+\mbf{q}(t) \\
		&=\!
		\int_{0}^{1}\!\srbmat{\mcl{T}_{\text{int}}\mbs{\psi}_{u,j}(t,s)+\mcl{T}_{\text{int},z}\mbs{\psi}_{z,j}(t,s)\\\mcl{T}_{\text{bf}}\mbs{\psi}_{u,j}(t,s)+\mcl{T}_{\text{bf},z}\mbs{\psi}_{z,j}(t,s)\\\mbs{\psi}_{z,j}(t,s)}ds 
		+ \srbmat{\mscr{D}_{\text{int}}\mbf{q}_{u}(t)\\\Lambda_{\text{bf}}\mbf{q}_{u}(t)\\\mbf{q}_{z}(t)} \\
		&=\!
		\int_{0}^{1}\!\srbmat{\mscr{D}_{\text{int}}(\mcl{T}\mbs{\psi}_{u,j}+\mcl{T}_{z}\mbs{\psi}_{z,j})(t,s)\\\Lambda_{\text{bf}}(\mcl{T}\mbs{\psi}_{u,j}+\mcl{T}_{z}\mbs{\psi}_{z,j})(t,s)\\\mbs{\psi}_{z,j}(t,s)}ds 
		+ \srbmat{(\mscr{D}_{\text{int}}\mbs{\phi}_{u,j})(0)\\(\Lambda_{\text{bf}}\mbs{\phi}_{u,j})(0)\\\mbs{\phi}_{z,j}(0)} \\
		&=\!
		\int_{0}^{1}\!\srbmat{(\mscr{D}_{\text{int}}\partial_{s}\mbs{\phi}_{u,j})(t,s)\\(\Lambda_{\text{bf}}\partial_{s}\mbs{\phi}_{u,j})(t,s)\\\partial_{s}\mbs{\psi}_{z,j}(t,s)}ds 
		\!+\! \srbmat{(\mscr{D}_{\text{int}}\mbs{\phi}_{u,j})(0)\\(\Lambda_{\text{bf}}\mbs{\phi}_{u,j})(0)\\\mbs{\phi}_{z,j}(0)}
		=\!
		\srbmat{(\mscr{D}_{\text{int}}\mbs{\phi}_{u,j})(t,1)\\ (\Lambda_{\text{bf}}\mbs{\phi}_{u,j})(t,1)\\ \mbs{\phi}_{z,j}(t,1)},
	\end{align*}
	proving that $(\mbs{\phi}_{j},\mbf{r})$ satisfies the $j$th PDE~\eqref{eq:delayPDE_transport}. Using a similar derivation, it also follows that for any solution $(\mbs{\phi}_{j},\mbf{r})$ to the $j$th PDE, with $\bl(\bar{\mcl{T}}_{v,j}\mbs{\psi}_{j}\br)(t,s)+T_{q}\mbf{q}(t)$, $(\mbs{\psi}_{j},\mbf{r})$ is also a solution to the PIE defined by $\mbf{G}_{\text{pie}}$.
\end{proof}

Having shown that both the PDE~\eqref{eq:delayPDE_PDE} and the PDE~\eqref{eq:delayPDE_transport} can be equivalently represented as PIEs, we now take the feedback interconnection of these systems to obtain a PIE representation of the delayed PDE.

\begin{cor}\label{cor:PIEmap_DPDE}[PIE Representation of DPDE]
	Let $\mbf{G}_{\text{pdde}}$ and $\mbf{G}_{\text{bc}}$ as in~\eqref{eq:DPDE_params} define a well-posed system of PDEs as in~\eqref{eq:delayPDE_PDE} and~\eqref{eq:delayPDE_transport}. Let $\mbf{G}_{\text{pie},1}$ denote the operators defining the PIE associated to the PDE~\eqref{eq:delayPDE_PDE}, as in Lemma~\ref{lem:PIEmap:PDE}, and let $\mbf{G}_{\text{pie},2}$ denote the operators defining the PIE associated to the PDE~\eqref{eq:delayPDE_transport}, as in Lemma~\ref{lem:PIEmap:PDE_transport}. Finally, let $\{\mcl{T},\mcl{T}_{z},\mcl{A},\mcl{B}_{z},\mcl{C}_{w},\mcl{D}_{wz}\}=\mbf{G}_{\text{pie}}=\mcl{L}_{\text{pie}\times\text{pie}}(\mbf{G}_{\text{pie},1},\mbf{G}_{\text{pie},2})$, where the linear operator map $\mcl{L}_{\text{pie}\times\text{pie}}$ is as defined in Prop.~\ref{prop:PIE_interconnection_full}. Then, $(\smallbmat{\mbf{v}\\\mbs{\psi}},w)$ is a solution to the PIE defined by $\mbf{G}_{\text{pie}}$ with initial conditions $(\mbf{v}_{0},\mbs{\psi}_{0})\in L_2^{\|\enn{p}\|}\times \text{Z}_{2}^{K(n_z,\|\enn{p}\|)}$ and input $z$ if and only if $(\smallbmat{\mbf{u}\\\mbs{\phi}},w)$ with $\smallbmat{\mbf{u}\\\mbs{\phi}}=\mcl{T}\smallbmat{\mbf{v}\\\mbs{\psi}}+\mcl{T}_{z}z$ is a solution to the delayed PDE defined by $\{\mbf{G}_{\text{pdde}},\mbf{G}_{\text{bc}},\mbs{\tau}\}$ with initial conditions $\smallbmat{\mbf{u}_{0}\\\mbs{\phi}_{0}}=\mcl{T}\smallbmat{\mbf{v}_{0}\\\mbs{\psi}_{0}}+\mcl{T}_{z}$.
\end{cor}


\begin{proof}
	By Lemma~\ref{lem:PIEmap:PDE}, $(\mbf{v},w,\mbf{q})$ is a solution to the PIE defined by $\{\mcl{T}_{1},\mcl{T}_{z,1},0,\mcl{A},\hdots\}=\mbf{G}_{\text{pie},1}$ with inputs $(z,\mbf{r})$ and initial conditions $\mbf{v}_{0}$ if and only if $(\mcl{T}_{1}\mbf{v}+\mcl{T}_{z,1}z,w,\mbf{q})$ is a solution to the PDE~\eqref{eq:delayPDE_PDE} defined by $\{\mbf{G}_{\text{pdde}},\mbf{G}_{\text{bc}}\}$ with initial conditions $\mcl{T}_1\mbf{v}_{0}+\mcl{T}_{z,1}z(0)$ and inputs $(z,\mbf{r})$.
	Similarly, by Lemma~\ref{lem:PIEmap:PDE_transport}, $(\mbs{\psi},\mbf{r})$ is a solution to the PIE defined by $\{\mcl{T}_{2},\mcl{T}_{q,2},\hdots\}=\mbf{G}_{\text{pie},2}$ with input $\mbf{q}$ and initial conditions $\mbs{\psi}_{0}$ if and only if $(\mcl{T}_{2}\mbs{\psi}+\mcl{T}_{q,2}\mbf{q},\mbf{r})$ is a solution to the PDE~\eqref{eq:delayPDE_transport} defined by $\{\mbs{\tau},\mbf{G}_{\text{bc}}\}$ with initial conditions $\mcl{T}_2\mbs{\psi}_{0}+\mcl{T}_{q,2}\mbf{q}(0)$ and input $\mbf{q}$. By Prop.~\ref{prop:PIE_interconnection_full}, it follows that $(\smallbmat{\mbf{v}\\\mbs{\psi}},w)$ is a solution to the PIE defined by $\mbf{G}_{\text{pie}}=\mcl{L}_{\text{pie}\times\text{pie}}(\mbf{G}_{\text{pie},1},\mbf{G}_{\text{pie},2})$ with initial conditions $\smallbmat{\mbf{v}_{0}\\\mbs{\psi}_{0}}$ if and only if $(\mbf{u},\mbf{q},w)$ and $(\mbs{\phi},\mbf{r})$ are solutions to the PDEs~\eqref{eq:delayPDE_PDE} and~\eqref{eq:delayPDE_transport} defined by $\{\mbf{G}_{\text{pdde}},\mbf{G}_{\text{bc}},\mbs{\tau}\}$, with initial conditions $\mbf{u}_{0}$ and $\mbs{\phi}_{0}$ and inputs $(z,\mbf{r})$ and $\mbf{q}$, where $\smallbmat{\mbf{u}\\\mbs{\phi}}=\mcl{T}\smallbmat{\mbf{v}\\\mbs{\psi}}+\mcl{T}_{z}z$ and $\smallbmat{\mbf{u}_{0}\\\mbs{\phi}_{0}}=\mcl{T}\smallbmat{\mbf{v}_{0}\\\mbs{\psi}_{0}}+\mcl{T}_{z}z$.
	
\end{proof}

%

\paragraph*{\textbf{Example}}
For the delayed PDE~\eqref{eq:example_DPDE}, we define
\begin{align*}
	\srbmat{\mbs{\phi}_{u}(t,s,x)\\\mbs{\phi}_{z}(t,s)}\!:=\!\srbmat{\mbf{u}(t\!-\!\tau s,x)\\z(t\!-\!\tau s)},\quad
	\srbmat{\mbs{\psi}_{u}(t,s,x)\\\mbs{\psi}_{z}(t,s)}\!:=\!\srbmat{\partial_{s}\partial_{x}^2\mbs{\phi}(t,s,x)\\\partial_{s}\mbs{\phi}_{z}(t,s)}.
\end{align*}
Letting $\mcl{T},\mcl{T}_{z}$ be as in~\eqref{eq:example_Tmap_DPDE}, the states $(\mbs{\phi}_{u},\mbs{\phi}_{z})$ must satisfy
\begin{align*}
	\srbmat{\mbs{\phi}_{u}(t,s)\\\mbs{\phi}_{z}(t,s)}&=\srbmat{\mbf{u}(t)\\z(t)}+\int_{0}^{s}\srbmat{(\partial_{s}\mbs{\phi}_{u})(t,\nu)\\(\partial_{s}\mbs{\phi}_{z})(t,\nu)}d\nu \\
	&=\srbmat{\mcl{T}\mbf{v}(t)+\mcl{T}_{z}z(t)\\z(t)}+\int_{0}^{s}\srbmat{\mcl{T}\mbs{\psi}_{u}(t,\nu)+\mcl{T}_{z}\mbs{\psi}_{z}(t,\nu)\\\mbs{\psi}_{z}(t,\nu)}d\nu.
\end{align*}
Imposing this relation, as well as $\mbf{u}=\mcl{T}\mbf{v}+\mcl{T}_{z}z$, the delayed PDE~\eqref{eq:example_DPDE} can be equivalently represented as a PIE
\begin{align*}
	&\mcl{T}_{z}\dot{z}(t)+\mcl{T}\mbf{v}_{t}(t)=\mbf{v}(t)+7\mcl{T}_{z}z(t)+7\mcl{T}\mbf{v}(t,x)\\
	&\hspace*{3.0cm} -3\int_{0}^{1}\bl[\mcl{T}_{z}\mbs{\psi}_{z}(t,s)+\mcl{T}\mbs{\psi}_{u}(t,s)\br]ds, \\
	&\srbmat{\mcl{T}\mbf{v}_{t}(t)+\mcl{T}_{z}\dot{z}(t)\\\dot{z}(t)}+\int_{0}^{s}\partial_{t}\srbmat{\mcl{T}\psi_{u}(t,\nu)+\mcl{T}_{z}\mbs{\psi}_{z}(t,\nu)\\\mbs{\psi}_{z}(t,\nu)}d\nu\\
	&\hspace*{3.75cm}=-\frac{1}{\tau}\srbmat{\mcl{T}\mbs{\psi}_{u}(t,s)+\mcl{T}_{z}\mbs{\psi}_{z}(t,\nu)\\\mbs{\psi}_{z}(t,s)}.
\end{align*}

\subsection{A PIE Representation of ODE-PDEs with Delays}\label{appx:PIEmap:ODEPDEdelay}

Having derived a PIE representation of both ODEs and PDEs with delay, we now take the interconnection of these PIEs, to derive a PIE representation of an ODE-PDE,
\begin{align}\label{eq:delayODEPDE}
	\srbmat{\dot{u}(t)\\z(t)}&=
	\left[\!{\small
		\begin{array}{ll}
			A&B_{w}\\
			C_{z}&0
	\end{array}}\!
	\right]
	\srbmat{u(t)\\w(t)} +
	\sum_{j=1}^{K}
	\left[\!{\small
		\begin{array}{ll}
			A_{j}	\\
			C_{z,j}
	\end{array}}\!
	\right]u(t-\tau_{j}), \\
	\srbmat{\partial_{t}\mbf{u}_{\text{p}}(t)\\w(t)}&=
	\left[\!{\small
		\begin{array}{lll}
			A_{\text{p}}& A_{\text{b}} & B_{z}\\
			\smallint_{a}^{b}[C_{w\text{p}}] & C_{w\text{b}} & B_{wz}
	\end{array}}\!	
	\right]
	\srbmat{\bl(\mscr{D}_{\text{int}}\mbf{u}_{\text{p}}\br)(t)\\\bl(\Lambda_{\text{bf}}\mbf{u}_{\text{p}}\br)(t)\\ z(t)} \\[-0.2em]
	&\hspace*{0.5cm}+\sum_{j=1}^{K}	
	\left[\!{\small
		\begin{array}{ll}
			A_{\text{p},j}& A_{\text{b},j}\\
			\smallint_{a}^{b}[C_{w\text{p},j}] & C_{w\text{b},j}
	\end{array}}\!	
	\right]\!
	\srbmat{\bl(\mscr{D}_{\text{int}}\mbf{u}_{\text{p}}\br)(t-\tau_{j})\\\bl(\Lambda_{\text{bf}}\mbf{u}_{\text{p}}\br)(t-\tau_{j})},	\notag\\[-0.4em]
	\text{with BCs}\hspace*{-0.1cm}&\qquad 0=\left[\!{\small
		\begin{array}{lll}
			\smallint_{a}^{b}[E_{\text{p}}] & E_{\text{b}} & E_{z}
		\end{array}\!}
	\right]
	\srbmat{\mscr{D}_{\text{int}}\mbf{u}_{\text{p}}(t)\\\Lambda_{\text{bf}}\mbf{u}_{\text{p}}(t)\\ z(t)}, \notag
\end{align}
defined by $\{\mbf{G}_{\text{dde}},\mbf{G}_{\text{pdde}},\mbf{G}_{\text{bc}}\}$ as before. 

\paragraph*{\textbf{Example}}
Taking the interconnection of the delayed ODE~\eqref{eq:example_DDE} with the delayed PDE~\eqref{eq:example_DPDE}, we obtain a system
\begin{align}\label{eq:example_DDE-DPDE}
	\dot{u}(t)&=-u(t)+u(t-\tau), \\
	\mbf{u}_{t}(t,x)&=\mbf{u}_{xx}(t,x)+10\mbf{u}(t,x)-3\mbf{u}(t-\tau,x),	&	x&\in\Omega_{0}^{1}, \notag\\
	\mbf{u}(t,0)&=0,\hspace*{1.0cm} \mbf{u}(t,1)=u(t). \notag
\end{align}

\begin{defn}[Solution to the DDE-DPDE]
	For given initial conditions $(u_{0},\mbs{\phi}_{0})\in\R^{n_u}\times \bar{Y}_{u_0}^{K}$ and $(\mbf{u}_{\text{p},0},\mbs{\phi}_{p,0})\in X_{w}^{\enn{p}}\times \bar{\mbf{Y}}_{\mbf{q}_0}^{K}$, where $\mbf{q}_{0}=\smallbmat{\mscr{D}_{\text{int}}\mbf{u}_{\text{p},0}(t)\\\Lambda_{\text{bf}}\mbf{u}_{\text{p},0}(t)\\C_{z}u_0+C_{z\text{d}}\mbs{\phi}_0(1)}$ we say that $(u,\mbs{\phi},\mbf{u}_{\text{p}},\mbs{\phi}_{\text{p}})$ is a solution to the delayed ODE-PDE system defined by $\{\mbf{G}_{\text{dde}},\mbf{G}_{\text{pdde}},\mbf{G}_{\text{bc}},\mbs{\tau}\}$ if $((u,\mbs{\phi}),z)$ is a solution to the ODE with delay defined by $\{\mbf{G}_{\text{dde}},\mbs{\tau}\}$ with initial conditions $(u_{0},\mbs{\phi}_{0})$ and input $w$, and $((\mbf{u}_{\text{p}},\mbs{\phi}_{\text{p}}),w)$ is a solution to the PDE with delay defined by $\{\mbf{G}_{\text{pdde}},\mbf{G}_{\text{bc}},\mbs{\tau}\}$ with initial conditions $(\mbf{u}_{\text{p},0},\mbs{\phi}_{\text{p},0})$ and input $z$.
\end{defn}

Having derived PIE representations associated to both delayed ODEs and delayed PDEs, a PIE representation for the delayed ODE-PDE interconnection~\eqref{eq:delayODEPDE} can be obtained by simply taking the interconnection of the PIE representations of each subsystem. This PIE will model the dynamics of a fundamental state $\mbf{v}\in\text{Z}_{12}^{(n_u,Kn_u,\|\enn{p}\|,K\|\enn{p}\|)}$, defined as \vspace*{-0.3cm}
\begin{align*}
	\mbf{v}=\srbmat{\hat{u}\\\mbs{\phi}\\\mbf{u}_{\text{p}}\\\mbs{\phi}_{\text{p}}}=
	\overbrace{\left[\!
		\begin{array}{llll}
			I\\&\partial_{s}\\&&\mscr{D}\\&&&\bar{\mscr{D}}_{v}
		\end{array}\!
		\right]}^{\bar{\mscr{D}}_{u}}
	\srbmat{u\\\mbs{\phi}\\\mbf{u}_{\text{p}}\\\mbs{\phi}_{\text{p}}}=\bar{\mscr{D}}_{u}\mbf{u},
\end{align*}
where $\mscr{D}$ and $\bar{\mscr{D}}_{v}$ are as in~\eqref{eq:Dmap_uPDE} and~\eqref{eq:Dmap_vPDE}, respectively.

\begin{cor}[PIE Representation of Delayed ODE-PDE]\label{cor:PIEmap_DDEDPDE}
	Let $\bar{\mbf{G}}_{\text{dde-pdde}}=\{\mbf{G}_{\text{dde}},\mbf{G}_{\text{pdde}},\mbf{G}_{\text{bc}},\mbs{\tau}\}$ define an ODE-PDE system with delay as in~\eqref{eq:delayODEPDE}. Let $\mbf{G}_{\text{pie},1}$ denote the parameters defining the PIE associated to the delayed ODE defined by $\{\mbf{G}_{\text{dde}},\mbs{\tau}\}$, as in Cor.~\ref{cor:PIEmap_DDE}. Let further $\mbf{G}_{\text{pie},2}$ denote the parameters defining the PIE associated to the delayed PDE defined by $\{\mbf{G}_{\text{pdde}},\mbf{G}_{\text{bc}},\mbs{\tau}\}$, as in Cor.~\ref{cor:PIEmap_DPDE}. Finally, let $\{\mcl{T},\mcl{A}\}=\mbf{G}_{\text{pie}}=\mcl{L}_{\text{pie}\times\text{pie}}(\mbf{G}_{\text{pie},1},\mbf{G}_{\text{pie},2})$, where the linear operator map $\mcl{L}_{\text{pie}\times\text{pie}}$ is as defined in Prop.~\ref{prop:PIE_interconnection_full}. Then, $\mbf{v}:=(\hat{u},\mbs{\psi},\mbf{v}_{\text{p}},\mbs{\psi}_{\text{p}})$ is a solution to the PIE defined by $\mbf{G}_{\text{pie}}$ with initial conditions $\mbf{v}_{0}=(\hat{u}_{0},\mbs{\psi}_{0},\mbf{v}_{\text{p},0},\mbs{\psi}_{\text{p},0})\in \R^{n_u}\times L_2^{Kn_u}[\Omega_{0}^{1}]\times L_2^{\|\enn{p}\|}[\Omega_{a}^{b}]\times L_{2}^{K\|\enn{p}\|}[\Omega_{0a}^{1b}]$ if and only if  $\mcl{T}\mbf{v}$
	is a solution to the ODE-PDE with delay defined by $\bar{\mbf{G}}_{\text{dde-pdde}}$ with initial conditions $\mcl{T}\mbf{v}_{0}$.
\end{cor}

\begin{proof}
	Let $\mbf{v}:=(\hat{u},\mbs{\psi},\mbf{v}_{\text{p}},\mbs{\psi}_{\text{p}})$ be an arbitrary solution to the PIE defined by $\mbf{G}_{\text{pie}}$ with initial conditions $\mbf{v}_{0}=(\hat{u}_{0},\mbs{\psi}_{0},\mbf{v}_{\text{p},0},\mbs{\psi}_{\text{p},0})\in \R^{n_u}\times L_2^{Kn_u}[\Omega_{0}^{1}]\times L_2^{\|\enn{p}\|}[\Omega_{a}^{b}]\times L_{2}^{K\|\enn{p}\|}[\Omega_{0a}^{1b}]$. By Prop.~\ref{prop:PIE_interconnection_full}, it follows that there exist signals $z$ and $w$ such that $\bl(\smallbmat{\hat{u}\\\mbs{\psi}},z\br)$ and $\bbl(\smallbmat{\mbf{v}_{\text{p}}\\\mbs{\psi}_{\text{p}}},w\bbr)$ are solutions to the PIEs defined by $\mbf{G}_{\text{pie},1}$ and $\mbf{G}_{\text{pie},2}$ with initial conditions $\smallbmat{\hat{u}_{0}\\\mbs{\psi}_{0}}$ and $\smallbmat{\mbf{v}_{\text{p},0}\\\mbs{\psi}_{\text{p},0}}$ and inputs $w$ and $z$, respectively. Let $\mbf{G}_{\text{pie},1}=\{\mcl{T}_{1},\mcl{T}_{1w},\hdots\}$ and $\mbf{G}_{\text{pie},2}=\{\mcl{T}_{2},\mcl{T}_{2z},\hdots\}$. Then, by Corollary~\ref{cor:PIEmap_DDE},  $\bl(\mcl{T}_{1}\smallbmat{\hat{u}\\\mbs{\psi}}+\mcl{T}_{1w}w,z\br)$ is a solution to the DDE defined by $\{\mbf{G}_{\text{dde}},\mbs{\tau}\}$ with initial conditions $\mcl{T}_{1}\smallbmat{\hat{u}_{0}\\\mbs{\psi}_{0}}+\mcl{T}_{1w}w(0)$ and input $w$. Similarly, by Corollary~\ref{cor:PIEmap_DPDE}, $\bbl(\mcl{T}_{2}\smallbmat{\mbf{v}_{\text{p}}\\\mbs{\psi}_{\text{p}}}+\mcl{T}_{2z}z,w\bbr)$ is a solution to the PDE with delay defined by $\{\mbf{G}_{\text{pdde}},\mbf{G}_{\text{bc}},\mbs{\tau}\}$ with initial conditions $\mcl{T}_{2}\smallbmat{\mbf{v}_{\text{p},0}\\\mbs{\psi}_{\text{p},0}}+\mcl{T}_{2z}z(0)$ and input $z$. By definition of the interconnection signals $w$ and $z$, as well as the PI operator $\mcl{T}$, we further have that
	\begin{align*}
		\mcl{T}\mbf{v}\!:=\!\srbmat{
			\mcl{T}_{1}\smallbmat{\hat{u}\\\mbs{\psi}}+\!\mcl{T}_{1w}w\\
			\mcl{T}_{2}\smallbmat{\mbf{v}_{\text{p}}\\\mbs{\psi}_{\text{p}}}+\!\mcl{T}_{2z}z}\!, & &\text{and} & &
		\mcl{T}\mbf{v}_{0}\!=\!\srbmat{\mcl{T}_{1}\smallbmat{\hat{u}_{0}\\\mbs{\psi}_{0}}+\!\mcl{T}_{1w}w(0)\\
			\mcl{T}_{2}\smallbmat{\mbf{v}_{\text{p},0}\\\mbs{\psi}_{\text{p},0}}+\!\mcl{T}_{2z}z(0)}\!.
	\end{align*}
	Combining these results, it follows that $\mcl{T}\mbf{v}$ is a solution to the ODE-PDE with delay defined by $\bar{\mbf{G}}_{\text{dde}-\text{pdde}}$ with initial conditions $\mcl{T}\mbf{v}_{0}$.
	
	Conversely, let now $\smallbmat{\mbf{u}_{1}\\\mbf{u}_{2}}:=\mcl{T}\smallbmat{\mbf{v}_{1}\\\mbf{v}_{2}}$ be a solution to the ODE-PDE with delay defined by $\bar{\mbf{G}}_{\text{dde}-\text{pdde}}$ with initial conditions $\smallbmat{\mbf{u}_{1,0}\\\mbf{u}_{2,0}}:=\mcl{T}\smallbmat{\mbf{v}_{1,0}\\\mbf{v}_{2,0}}$. Then, there exist interconnection signals $z$ and $w$ such that $(\mbf{u}_{1},z)$ is a solution to the DDE defined by $\{\mbf{G}_{\text{dde}},\mbs{\tau}\}$ with initial conditions $\mbf{u}_{1}$ and input $w$, and $(\mbf{u}_{2},w)$ is a solution to the delayed PDE defined by $\{\mbf{G}_{\text{pdde}},\mbf{G}_{\text{bc}},\mbs{\tau}\}$ with initial conditions $\mbf{u}_{2}$ and input $z$. Letting $\mbf{G}_{\text{pie},1}=\{\mcl{T}_{1},\mcl{T}_{1w},\hdots\}$, by definition of the operator $\mcl{T}$ and Corollary~\ref{cor:PIEmap_DDE}, it follows that $\mbf{v}_{1}$ is a solution to the PIE defined by $\mbf{G}_{\text{pie},1}$ with initial conditions $\mbf{v}_{1,0}$ and input $w$. Similarly, letting $\mbf{G}_{\text{pie},2}=\{\mcl{T}_{2},\mcl{T}_{2z},\hdots\}$, by definition of the operator $\mcl{T}$ and Corollary~\ref{cor:PIEmap_DPDE}, it follows that $\mbf{v}_{2}$ is a solution to the PIE defined by $\mbf{G}_{\text{pie},2}$ with initial conditions $\mbf{v}_{2,0}$ and input $z$. Finally, by Proposition~\ref{prop:PIE_interconnection_full}, $\mbf{v}:=\smallbmat{\mbf{v}_{1}\\\mbf{v}_{2}}$ is a solution to the PIE defined by $\mbf{G}_{\text{pie}}$ with initial conditions $\smallbmat{\mbf{v}_{1,0}\\\mbf{v}_{2,0}}$.
\end{proof}

\paragraph*{\textbf{Example}}
Letting $\mbf{v}(t,x)=\partial_{x}^2\mbf{u}(t,x)$, $\mbs{\psi}_{z}(t,s)=\partial_{s}\mbs{\phi}_{z}(t,s)$ and $\mbs{\psi}_{u}(t,s,x)=\partial_{s}\partial_{x}^2\mbs{\phi}_{u}(t,s,x)$, where $\mbs{\phi}_{z}(t,s)=u(t-\tau s)$ and $\mbs{\phi}_{u}(t,s,x)=\mbf{u}(t-\tau s,x)$, the delayed ODE-PDE~\eqref{eq:example_DDE-DPDE} can be equivalently represented as a PIE as

	\begin{align*}
		&\dot{u}(t) = \int_{0}^{1}\mbs{\psi}_{z}(t,s)ds	\\
		&\dot{u}(t)+\!\!\int_{0}^{s}\!\partial_{t}\mbs{\psi}_{z}(t,\nu)d\nu
		\!=\!-\frac{1}{\tau}\mbs{\psi}_{z}(t,s), \\
		&\mcl{T}_{z}\dot{u}(t)\!+\!\mcl{T}\mbf{v}_{t}(t)	\\
		&\hspace*{0.5cm}=7\mcl{T}_{z}u(t)\!+\![1\!+\!7\mcl{T}]\mbf{v}(t)\!-\!3\!\int_{0}^{1}\![\mcl{T}_{z}\mbs{\psi}_{z}(t,\!s)\!+\!\mcl{T}\mbs{\psi}_{u}(t,\!s)]ds, \notag\\
		&\mcl{T}_{z}\dot{u}(t)\!+\!\mcl{T}\mbf{v}_{t}(t)\!+\!\!\!\int_{0}^{s}\!\!\![\partial_{t}\mcl{T}_{z}\mbs{\phi}_{z}(t,\!\nu)\!+\!\partial_{t}\mcl{T}\mbs{\psi}_{u}(t,\!\nu)]d\nu 	\\
		&\hspace*{3.0cm}=-\frac{[\mcl{T}_{z}\mbs{\phi}_{z}(t,\!s)\!+\!\mcl{T}\mbs{\psi}_{u}(t,\!s)]}{\tau}.		\notag
	\end{align*}

Using PIETOOLS, applying Thm.~\ref{thm:stability_as_LPI}, stability of this system can be verified for any delay $\tau<0.15$.

\end{appendices}
\end{document}